\documentclass{amsart}

\usepackage{amssymb}

\usepackage{graphicx}

\usepackage{cite}
\usepackage{psfrag}

\def\FF{\mathbb F}                         
\def\CC{\mathbb C}                         
\def\RR{\mathbb R}                         
\def\KK{\mathbb K}                         
\def\FFq2{{\mathbb F}_{q^2}}
\def\rank{\mathop{\mathrm {rk}}\nolimits}  
\def\diag{\mathop{\mathrm {diag}}\nolimits}  
\def\tr{\mathop{\top}\nolimits}     
     
\def\N{\mathop{\mathrm{N}}\nolimits}     
\def\hgl{{\mathop{\mathcal{HGL}}\nolimits}_n(\mathbb{F}_{q^2})}
\def\hgldva{{\mathop{\mathcal{HGL}}\nolimits}_2(\mathbb{F}_{q^2})}

\def\h{{\mathop{\mathcal{H}}\nolimits}_n(\mathbb{F}_{q^2})}
\def\hdva{{\mathop{\mathcal{H}}\nolimits}_2(\mathbb{F}_{q^2})}
\def\hglnminus1{{\mathop{\mathcal{HGL}}\nolimits}_{n-1}(\mathbb{F}_{q^2})}
\def\hminus1{{\mathop{\mathcal{H}}\nolimits}_{n-1}(\mathbb{F}_{q^2})}
\def\her{{\mathop{\mathcal{H}}\nolimits}_n(\mathbb{F}_{q^2})}

\copyrightinfo{}{}

\newtheorem{thm}{Theorem}[section]
\newtheorem*{main}{Main Theorem}
\newtheorem{lemma}[thm]{Lemma}
\newtheorem{cor}[thm]{Corollary}

\theoremstyle{definition}

\theoremstyle{remark}
\newtheorem{remark}[thm]{Remark}

\numberwithin{equation}{section}

\newcommand{\cal}{\mathcal}

\newcommand{\inv}[1]{{\overline{#1}}}

\newcounter{myenumi}

\newsavebox{\cmm}
\savebox{\cmm}{\indent}
\newenvironment{myenumerate}[1]{
\begin{list}{
{\bf #1~\themyenumi}. } {\labelwidth=0pt
\labelsep=0pt\leftmargin=0pt\usecounter{myenumi}} }{\end{list}}

\savebox{\cmm}{\indent}
\newenvironment{myenumerate2}[1]{
\begin{list}{
{\emph #1~\themyenumi}. } {\labelwidth=0pt
\labelsep=0pt\leftmargin=0pt\usecounter{myenumi}} }{\end{list}}

\begin{document}

\title{Adjacency preservers on invertible hermitian matrices~II}

\author[M.~Orel]{Marko Orel}
\address{University of Primorska, FAMNIT, Glagolja\v{s}ka 8, 6000 Koper, Slovenia}
\address{IMFM, Jadranska 19, 1000 Ljubljana, Slovenia}
\address{University of Primorska, IAM, Muzejski trg 2, 6000 Koper, Slovenia}

\email{marko.orel@upr.si}


\begin{abstract}
Maps that preserve adjacency on the set of all invertible hermitian matrices over a finite field are characterized.
It is shown that such maps form a group that is generated by the maps $A\mapsto PAP^{\ast}$, $A\mapsto A^{\sigma}$, and $A\mapsto A^{-1}$, where $P$ is an invertible matrix, $P^{\ast}$ is its conjugate transpose, and~$\sigma$ is an automorphism of the underlying field. Bijectivity of maps is not an assumption but a conclusion.
Moreover, adjacency is assumed to be preserved in one directions only.

The main result and author's previous result~\cite{FFA} are applied to characterize maps that preserve the `speed of light' on (a) finite Minkowski space-time and (b) the complement of the light cone in finite Minkowski space-time.
\end{abstract}

\maketitle

\section{Introduction}

Two hermitian matrices $A$ and $B$ are adjacent if the rank $\rank(A-B)$ equals one. In previous paper~\cite{prvi_del} the author shows that any map on the set $\hgl$ of all $n\times n$ invertible hermitian matrices over a finite field $\FFq2$, which maps adjacent matrices into adjacent matrices, is necessarily bijective.
In the language of graph theory this means that the graph, with vertex set $\hgl$ and edges defined by the adjacency relation, is a core, so any its endomorphism is an automorphism. In this paper we characterize all such maps.

In the case of all (singular and invertible) hermitian matrices, the characterization of bijective maps $\Phi$ that preserve adjacency in both directions, that is
\begin{equation}\label{ee7}
\rank\big(\Phi(A)-\Phi(B)\big)=1\Longleftrightarrow \rank(A-B)=1,
\end{equation}
is known as Hua's fundamental theorem of geometry of hermitian matrices. This kind of results for various matrix spaces are summarized in the book~\cite{wan}. For the set of all invertible hermitian matrices such a result is not known yet. We derive it for finite field case and then apply~\cite[Theorem~1.1]{prvi_del} to obtain a characterization of all adjacency preservers on $\hgl$. We also obtain two new results related to maps that preserve the `speed of light' on the finite field analog of Minkowski space-time, i.e., the geometrical space of special relativity.

We now state the main results of this paper.

\begin{thm}\label{thm1}
Let $n\geq 2$ be an integer and $q\geq 4$ a power of a prime. A bijective map $\Phi: \hgl\to \hgl$ preserves adjacency in both directions if and only if it is of the form
\begin{equation}\label{ii10}
\Phi(A)=PA^{\sigma}P^{\ast}\qquad \textrm{or}\qquad \Phi(A)=P(A^{\sigma})^{-1}P^{\ast},
\end{equation}
where $P$ is an invertible matrix over $\FFq2$, $P^{\ast}$ is its conjugate transpose, and $\sigma: \FFq2\to\FFq2$ is a field automorphism that is applied entry-wise to $A$.
\end{thm}

If a map $\Phi$ on $\hgl$ is bijective, then condition~(\ref{ee7}) is equivalent to
\begin{equation}\label{ee6}
\rank\big(\Phi(A)-\Phi(B)\big)=1\Longrightarrow \rank(A-B)=1,
\end{equation}
since the set $\hgl$ is finite. Moreover, any map that satisfies~(\ref{ee6}) for all $A,B\in\hgl$ is automatically bijective by Theorem~1.1 from previous paper~\cite{prvi_del}. Consequently, the above Theorem~\ref{thm1} proves the following result.
\begin{main}
Let $n\geq 2$ be an integer and $q\geq 4$ a power of a prime. A map $\Phi: \hgl\to \hgl$ preserves adjacency if and only if it is of the form
\begin{equation}\label{ee8}
\Phi(A)=PA^{\sigma}P^{\ast}\qquad \textrm{or}\qquad \Phi(A)=P(A^{\sigma})^{-1}P^{\ast}
\end{equation}
for some invertible matrix $P$ and field automorphism $\sigma$.
\end{main}

\begin{remark}
In terms of graph theory, Main Theorem says that endomorphisms of the graph with vertex set $\hgl$ and edge set $\big\{\{A,B\} : \rank(A-B)=1\big\}$ are precisely the maps in~(\ref{ee8}).
\end{remark}

Detailed definitions, notation, and auxiliary results are described in Section~2. In Section~3 we prove Theorem~\ref{thm1}. Section~4 is devoted to Minkowski space-time. Besides the two new results on characterization of maps that preserve the `speed of light' on finite Minkowski space, we write a survey of few related results and describe the connection with special relativity. For a detailed state of art on adjacency preservers in general we refer to the introduction of previous paper~\cite{prvi_del}.

\section{Notation and auxiliary theorems}\label{2}

Let $\inv{\phantom{a}}$ be an \emph{involution} on a finite field, that is, $\inv{x+y}=\inv{x}+\inv{y}$, $\inv{xy}=\inv{y}\cdot \inv{x}$, and $\inv{\inv{x}}=x$. In this paper we assume that the involution is not the identity map, which implies that the finite field has $q^2$ elements, where $q$ is a power of a prime. We denote the finite field by $\FFq2$ and recall that the unique involution is given by the rule $\inv{x}=x^q$. We use Greek letters for elements of the \emph{fixed field} of the involution $\FF:=\{\lambda\in \FFq2\, :\, \inv{\lambda}=\lambda\}$, which has $q$ elements.

Let $n\geq 2$ be an integer. A $n\times n$ matrix $A$ with coefficients in $\FFq2$ is \emph{hermitian}, if $A^{\ast}:=\inv{A}^{\tr}=A$, where the involution $\inv{\phantom{a}}$ is applied entry-wise and $B^\top$ denotes the transpose of $B$. We use $\h$ and $\hgl$ to denote the set of all hermitian and the set of all invertible hermitian matrices respectively.
If the rank $\rank A$ of a hermitian matrix equals $r$, then
\begin{equation}\label{i8}
A=P(E_{11}+E_{22}+\ldots+E_{rr})P^{\ast}
\end{equation}
for some invertible matrix $P$, where $E_{ii}$ denotes the matrix with 1 at  $(i,i)$-th entry and zeros elsewhere~\cite[Theorem~4.1]{bose}. Consequently, $A=\sum_{i=1}^{r} {\bf x}_i{\bf x}_i^{\ast}$, where the column vector ${\bf x}_i\in\FFq2^n$ is the $i$-th column of matrix $P$. Two hermitian matrices $A$ and $B$ are \emph{adjacent} if $\rank(A-B)=1$. In that case, the unique maximal set of pairwise adjacent matrices, containing both $A$ and $B$, equals
\begin{equation}\label{ii1}
\{A+\lambda{\bf x}{\bf x}^{\ast}\ :\ \lambda\in\FF\},
\end{equation}
where ${\bf x}{\bf x}^{\ast}=B-A$~\cite[Corollary~6.9]{wan}. Moreover, $|\rank A - \rank B|\leq 1$.
Given a subset $U\subseteq \h$, a map $\Phi : U\to U$
\emph{preserves adjacency} if condition~(\ref{ee6}) is satisfied for $A,B\in U$. If $\Phi$ obeys the stronger rule~(\ref{ee7}), then $\Phi$ \emph{preserves adjacency in both directions}. If $\Phi$ is bijective, then preserving adjacency in one or both directions is equivalent, since the set $U$ is finite. The following result was proved in~\cite[Theorem~3.1]{FFA}.
\begin{lemma}\label{FFA}
A map $\Phi : \h\to\h$ preserves adjacency if and only if it is of the form $\Phi(A)=PA^{\sigma}P^{\ast}+B$ for some invertible matrix $P$, hermitian matrix~$B$, and field automorphism $\sigma: \FFq2\to\FFq2$ that is applied entry-wise to $A$.
\end{lemma}
By~(\ref{i8}), any $A\in\h$ of rank $r$ is of the form $A=P(\dot{A}\oplus 0)P^{\ast}$ for some invertible $P$ and $\dot{A}\in \mathcal{HGL}_{r}(\FFq2)$. Lemma~\ref{pomozna} is a modification of~\cite[Lemma~3.1]{lama}.
\begin{lemma}\label{pomozna}
Let $A=P(\dot{A}\oplus 0)P^{\ast}$ with $P$ invertible and $\dot{A}\in \mathcal{HGL}_{r}(\FFq2)$. If ${\bf x}=P(y_1,\ldots,y_n)^{\tr}$, $y_i\neq 0$ for some $i>r$, and $\lambda\neq 0$, then $\rank(A+\lambda {\bf x}{\bf x}^{\ast})=r+1$.
\end{lemma}
\begin{proof}
We may assume that $P=I$ is the identity matrix. Since $\rank (\lambda {\bf x}{\bf x}^{\ast})=1$, a row reduction performed with $i$-th row annihilates all other rows. The same row reduction  performed on $A+\lambda {\bf x}{\bf x}^{\ast}$ proves the claim.
\end{proof}
\begin{cor}\label{nova}
Let $1\leq r\leq n$ and
$A=P(\dot{A}\oplus 0)P^{\ast}$, where $P$ is invertible and $\dot{A}\in \mathcal{HGL}_{r}(\FFq2)$. If
${\bf x}=P(y_1,\ldots,y_r,0\ldots,0)^{\tr}$,  then $\rank(A+\lambda {\bf x}{\bf x}^{\ast})\in \{r-1, r\}$ for all $\lambda\in\FF$ and $\rank(A+\lambda {\bf x}{\bf x}^{\ast})=r-1$ for at most one scalar $\lambda$.
\end{cor}
\begin{proof}
It is obvious that $\rank(A+\lambda {\bf x}{\bf x}^{\ast})\in \{r-1, r\}$. Assume that $\rank B=r-1$, where $B:=A+\lambda_0 {\bf x}{\bf x}^{\ast}$. We need to show that for $\lambda\neq \lambda_0$ the corresponding matrix is of rank $r$. We may assume that $r\geq 2$. In that case $B=Q(\dot{B}\oplus 0)Q^{\ast}$ for some invertible $Q$ and $\dot{B}\in \mathcal{HGL}_{r-1}(\FFq2)$. Let $Q^{-1}{\bf x}=:{\bf z}=(z_1,\ldots,z_n)^{\tr}$. Since $A=B-\lambda_0 {\bf x}{\bf x}^{\ast}$ is of rank $r$, it follows that $z_{i}\neq 0$ for some $i> r-1$, so Lemma~\ref{pomozna} shows that $\rank(A+(\lambda_0-\lambda){\bf x}{\bf x}^{\ast})=\rank(B-\lambda{\bf x}{\bf x}^{\ast})=r$ for all nonzero $\lambda$.
\end{proof}
If $A$ and ${\bf x}$ are as in Lemma~\ref{pomozna}, then the set ${\cal L}:=\{A+\lambda{\bf x}{\bf x}^{\ast}\ :\ 0\neq\lambda\in\FF\}$, that consist of $q-1$ pairwise adjacent matrices of rank $r+1$, is a \emph{leaf} of $A$. The set of all leaves of $A$ is a \emph{flower} of $A$ (cf. Figure~1). If $\rank A =n-1$, then it can be easily deduced from Lemma~\ref{pomozna}, that its flower consists of $q^{2(n-1)}$ leaves.
\begin{lemma}\label{list_v_list}
Let $q\geq 3$, $1\leq r\leq n$, and assume $U\subseteq \h$ is a set of all matrices of rank $r$. If ${\cal L}\subseteq U$ is a leaf of some hermitian matrix of rank $r-1$ and  a bijective map $\Phi : U\to U$ preserves adjacency in both directions, then $\Phi({\cal L})$ is a leaf of some unique hermitian matrix of rank $r-1$.
\end{lemma}
\begin{proof}
Since $\Phi$ preserves adjacency, $\Phi({\cal L})$ contains $q-1\geq 2$ pairwise adjacent matrices of rank $r$. By~(\ref{ii1}), the maximal set of pairwise adjacent matrices that contains these $q-1$ matrices is unique and it contains an additional matrix $B\in\h$ with $\rank B\in \{r-1,r\}$. If $B\in U$, then $B=\Phi(C)$ for some $C\in U$. Since $\Phi$ preserves adjacency in both directions, $C$ is adjacent to all matrices in ${\cal L}$, a contradiction since ${\cal L}$ is a leaf. Hence, $\rank B=r-1$ and  $\Phi({\cal L})$ is its leaf.
\end{proof}

Since $\rank(A^{-1}-B^{-1})=\rank(B^{-1}(B-A)A^{-1})$ for invertible matrices, we deduce that the bijective map $A\mapsto A^{-1}$ preserves adjacency in both directions on $\hgl$. Consequently, if $q\geq 3$, $\rank A =n-1$, and ${\cal L}=\{A+\lambda{\bf x}{\bf x}^{\ast}\ :\ 0\neq\lambda\in\FF\}$ is a leaf of $A$, then Lemma~\ref{list_v_list} shows that ${\cal L}^{-1}:=\{(A+\lambda{\bf x}{\bf x}^{\ast})^{-1}\ :\ 0\neq\lambda\in\FF\}$ is a leaf of unique matrix of rank $n-1$, which is described in Lemma~\ref{lemma6}.

Few notions from graph theory are needed. All graphs in this paper are finite, undirected, and without loops and multiple edges. We use $V(\Gamma)$ and $E(\Gamma)$ to denote the vertex set and the edge set of graph $\Gamma$ respectively. A graph $\Gamma'$ is a \emph{subgraph} of graph $\Gamma$ if $V(\Gamma')\subseteq V(\Gamma)$ and $E(\Gamma')\subseteq E(\Gamma)$. A subgraph $\Gamma'$ is \emph{induced} by the set $U\subseteq V(\Gamma)$ if $U=V(\Gamma')$ and $E(\Gamma')=\{\{u,v\}\in E(\Gamma)\ :\ u,v\in U\}$.
With a slight abuse of notation, we denote the graph with vertex set $\h$ and edge set $\{\{A,B\}\ :\ A,B\in \h, \rank(A-B)=1\}$ still by $\h$. Similarly, we use $\hgl$ also to denote the subgraph in $\h$, which is induced by the set $\hgl$.
\begin{figure}[h!]
\centering
\includegraphics[width=0.5\textwidth]{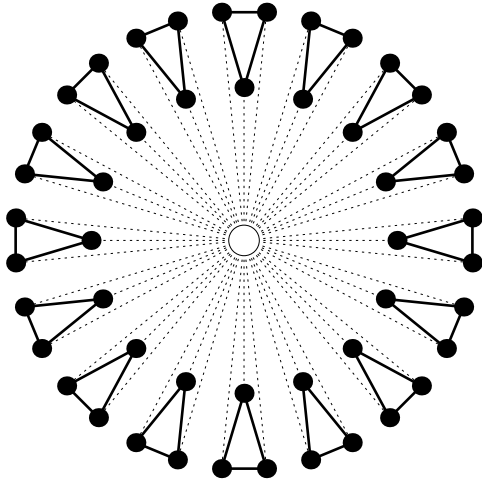}
\caption{A flower of a rank-one matrix in ${\cal HGL}_2(\FF_{4^2})$. Black vertices are invertible matrices, the white vertex is a rank-one matrix. Thick edges are in ${\cal HGL}_2(\FF_{4^2})$, dotted edges are in  ${\cal H}_2(\FF_{4^2})$.}
\end{figure}
For $q\geq 4$ the next lemma is proved in~\cite{prvi_del} as a corollary of a more complicated result with long proof. For the sake of completeness we present a short proof which works for $q=3$ as well.
\begin{lemma}\label{lemma_povezan}
Let $q\geq 3$. The graph $\hgl$ is connected.
\end{lemma}
\begin{proof}
Let $A=\sum_{i=1}^n {\bf x}_i{\bf x}_i^{\ast}$ be any matrix in $\hgl$ that is distinct from the identity matrix $I_n=\sum_{i=1}^n {\bf e}_i{\bf e}_i^{\ast}$. Here ${\bf e}_i$ is the $i$-th member of the standard basis.
We will construct a chain $I_n=A_0,A_1,A_2,\ldots,A_m=A$ made of matrices in $\hgl$ that connects  $I_n$ with $A$, that is, $\rank(A_i-A_{i+1})\leq 1$ for all $i$.
By Corollary~\ref{nova} and $q\geq 3$ there exists $\lambda_1\in \FF\backslash\{0\}$ such that $A_1:=A_0+\lambda_1{\bf x}_1{\bf x}_1^{\ast}$ is invertible. Choose $n-1$ vectors ${\bf e}_{i_1},\ldots,{\bf e}_{i_{n-1}}$ from ${\bf e}_1,\ldots,{\bf e}_n$ such that ${\bf x}_1,{\bf e}_{i_1},\ldots,{\bf e}_{i_{n-1}}$ form a basis of $\FF_{q^2}^{n}$. Then $A_2:=\sum_{k=1}^{n-1}{\bf e}_{i_k}{\bf e}_{i_k}^{\ast}+\lambda_1{\bf x}_1{\bf x}_1^{\ast}$ and $A_3:=\sum_{k=1}^{n-1}{\bf e}_{i_k}{\bf e}_{i_k}^{\ast}+{\bf x}_1{\bf x}_1^{\ast}$ are invertible.
We now repeat the procedure. By Corollary~\ref{nova} there is $\lambda_2\in \FF\backslash\{0\}$ such that $A_4:=A_3+\lambda_2{\bf x}_2{\bf x}_2^{\ast}$ is invertible. Choose $n-2$ vectors ${\bf e}_{j_1},\ldots,{\bf e}_{j_{n-2}}$ from ${\bf e}_{i_1},\ldots,{\bf e}_{i_{n-1}}$  such that ${\bf x}_1,{\bf x}_2,{\bf e}_{j_1},\ldots,{\bf e}_{j_{n-2}}$ form a basis of $\FF_{q^2}^{n}$. Then $A_5:=\sum_{k=1}^{n-2}{\bf e}_{j_k}{\bf e}_{j_k}^{\ast}+{\bf x}_1{\bf x}_1^{\ast}+\lambda_2{\bf x}_2{\bf x}_2^{\ast}$ and $A_6:=\sum_{k=1}^{n-2}{\bf e}_{j_k}{\bf e}_{j_k}^{\ast}+{\bf x}_1{\bf x}_1^{\ast}+{\bf x}_2{\bf x}_2^{\ast}$ are invertible. By repeating the procedure we obtain the desired chain.
\end{proof}
The distance in $\h$ between $A$ and $B$ is given by $\rank(A-B)$~\cite[Proposition~6.5]{wan}. In this paper $d$ denotes the distance in $\hgl$. Obviously, $d(A,B)\geq \rank (A-B)$ for all $A,B\in\hgl$. Any bijective map $\Phi: \hgl\to\hgl$ that preserves adjacency in both directions is an automorphism of the corresponding graph, so $d\big(\Phi(A),\Phi(B)\big)=d(A,B)$ for all $A,B\in \hgl$.

\section{Proofs}

Our first three results essentially describe the action of the map $A\mapsto A^{-1}$ on a leaf/flower made of invertible matrices.

\begin{lemma}\label{lemma6}
\begin{enumerate}
\item
Let ${\cal L}=\{A+\lambda{\bf x}{\bf x}^{\ast}\ :\ 0\neq\lambda\in\FF\}$ be a leaf of a matrix~(\ref{i8}) with $r=n-1$, and denote ${\bf x}=P(y_1,\ldots,y_n)^{\tr}$ with $y_n\neq 0$. Then
\begin{equation}\label{eq5}
(A+\lambda {\bf x}{\bf x}^{\ast})^{-1}=
(P^{-1})^{\ast}\begin{bmatrix}
I_{n-1}& {\bf z}\\
{\bf z}^{\ast}& {\bf z}^{\ast}{\bf z}+(\lambda y_n\inv{y}_n)^{-1}
\end{bmatrix}P^{-1}
\end{equation}
for all  $\lambda\in \FF\backslash\{0\}$, where ${\bf z}:=\big(\frac{-y_1}{y_n},\ldots,\frac{-y_{n-1}}{ y_n}\big)^{\tr}$. If $q\geq 3$, ${\cal L}^{-1}$ is a leaf of matrix $(P^{-1})^{\ast}\left[\begin{smallmatrix}
I_{n-1}& {\bf z}\\
{\bf z}^{\ast}& {\bf z}^{\ast}{\bf z}
\end{smallmatrix}\right]P^{-1}$, which is generated by vector $(P^{-1})^{\ast}{\bf e}_n$. Here ${\bf e}_n:=(0,\ldots,0,1)^{\tr}$ and
$I_{n-1}$ is the $(n-1)\times (n-1)$ identity matrix.

\item Let $D\in\hglnminus1$ and ${\bf w}\in\FFq2^{n-1}$. Then
$\left[\begin{smallmatrix}
\mu & {\bf w}^{\ast}\\
{\bf w} & D
\end{smallmatrix}\right]$
is invertible if and only if $\mu\neq {\bf w}^{\ast} D^{-1}{\bf w}$, in which case
    \begin{equation}\label{ii2}
    \begin{bmatrix}
\mu & {\bf w}^{\ast}\\
{\bf w} & D
\end{bmatrix}^{-1}=\begin{bmatrix}
0&0\\
0&D^{-1}
\end{bmatrix}+\frac{1}{\mu-{\bf w}^{\ast}D^{-1}{\bf w}}
\begin{bmatrix}
-1\\
D^{-1}{\bf w}
\end{bmatrix}
\begin{bmatrix}
-1\\
D^{-1}{\bf w}
\end{bmatrix}^{\ast}.
\end{equation}
\end{enumerate}
\end{lemma}
\begin{proof}
(i) A straightforward calculation shows that $(A+\lambda {\bf x}{\bf x}^{\ast})B=I_n$, where $B$ is the matrix in~(\ref{eq5}). The rest is obvious since $\rank \left[\begin{smallmatrix}
I_{n-1}& {\bf z}\\
{\bf z}^{\ast}& {\bf z}^{\ast}{\bf z}
\end{smallmatrix}\right]=n-1$.
(ii) If $\mu\neq {\bf w}^{\ast} D^{-1}{\bf w}$, a straightforward calculation shows that $\left[\begin{smallmatrix}
\mu & {\bf w}^{\ast}\\
{\bf w} & D
\end{smallmatrix}\right]C=I_n$, where $C$ is the matrix in~(\ref{ii2}). By~(\ref{i8}), there is invertible $Q$ such that $D=QQ^{\ast}$. So if ${\bf z}:= Q^{-1}{\bf w}$, then
$\left[\begin{smallmatrix}
{\bf w}^{\ast} D^{-1}{\bf w} & {\bf w}^{\ast}\\
{\bf w} & D
\end{smallmatrix}\right]=
\left[\begin{smallmatrix}
1 & 0\\
0 & Q
\end{smallmatrix}\right]
\left[\begin{smallmatrix}
{\bf z}^{\ast}{\bf z} & {\bf z}^{\ast}\\
{\bf z} & I_{n-1}
\end{smallmatrix}\right]
\left[\begin{smallmatrix}
1 & 0\\
0 & Q^{\ast}
\end{smallmatrix}\right]$ is of rank $n-1$.
\end{proof}

\begin{cor}\label{lemma22}
Let $q\geq 3$. If ${\cal L}_1\neq {\cal L}_2$ are two leaves of matrix $A$, with $\rank A= n-1$, and $N_i$ is the matrix of rank $n-1$ of the leaf ${\cal L}_i^{-1}$, then
\begin{equation}\label{ii3}
N_i=Q\begin{bmatrix}
I_{n-1}& {\bf z}_i\\
{\bf z}_i^{\ast}& {\bf z}_i^{\ast}{\bf z}_i
\end{bmatrix}Q^{\ast}\qquad (i=1,2)
\end{equation}
for some invertible matrix $Q$
and column vectors ${\bf z}_1\neq {\bf z}_2$, so $\rank(N_1-N_2)=2$.
\end{cor}
\begin{proof}
Pick  ${\bf x}_i$ such that ${\cal L}_i=\{A+\lambda{\bf x}_i{\bf x}_i^{\ast}\ :\ 0\neq\lambda\in\FF\}$.
By Lemma~\ref{lemma6},   $N_i$ is as in~(\ref{ii3}). Since ${\cal L}_1\neq {\cal L}_2$, ${\bf x}_1$ and ${\bf x}_2$ are linearly independent, so ${\bf z}_1\neq {\bf z}_2$.
\end{proof}

\begin{cor}\label{cor2}
Let  $q\geq 3$ and let $N_1\neq N_2$ be hermitian and of rank $n-1$. Then there are leaves ${\cal L}_1$ of $N_1$ and ${\cal L}_2$ of $N_2$ such that ${\cal L}_1^{-1}$ and ${\cal L}_2^{-1}$ are leaves of a common matrix of rank $n-1$ iff~(\ref{ii3}) holds for some invertible  $Q$ and  ${\bf z}_1\neq {\bf z}_2$.
\end{cor}
\begin{proof}
Assume that there are such ${\cal L}_1$ and ${\cal L}_2$. Then ${\cal L}_1^{-1}$ and ${\cal L}_2^{-1}$ are leaves of some $A$. Since ${\cal L}_i= ({\cal L}_i^{-1})^{-1}$, (\ref{ii3}) holds by Corollary~\ref{lemma22}.

Contrary, if (\ref{ii3}) holds, then pick ${\cal L}_i:=\{N_i+\lambda QE_{nn}Q^{\ast}\ :\ 0\neq\lambda\in\FF\}$. From (\ref{eq5}) it follows that ${\cal L}_1^{-1}$ and ${\cal L}_2^{-1}$ are both leaves of $A=(Q^{-1})^{\ast}(I_{n-1}\oplus 0)Q^{-1}$.
\end{proof}

The following lemma is the first big step toward the proof of Theorem~\ref{thm1}.

\begin{lemma}\label{lemma7}
Assume $q\geq 4$, $\Phi: \hgl\to \hgl$ is a bijection that preserves adjacency in both directions, ${\cal L}_1\neq {\cal L}_2$ are leaves of a common matrix of rank $n-1$, and $\Phi({\cal L}_1), \Phi({\cal L}_2)$ are not leaves of a common matrix.  Then $\Phi({\cal L}_1)^{-1}, \Phi({\cal L}_2)^{-1}$ are leaves of a common matrix. Moreover, if $\Phi({\cal L}_i)=\{M_i+\lambda {\bf x}_i{\bf x}_i^{\ast}\ :\ 0\neq \lambda\in \FF\}$ with $\rank M_i=n-1$, then ${\bf x}_1, {\bf x}_2$ are linearly dependent and $\rank(M_1-M_2)=2$.
\end{lemma}

\begin{proof}
For $i=1,2$ write ${\cal L}_i=\{N+\lambda{\bf y}_i{\bf y}_i^{\ast}\ :\  0\neq\lambda\in \FF\}$, where $\rank N=n-1$.
There exist bijections $f_1$ and $f_2$ on $\FF\backslash \{0\}$, column vectors ${\bf x}_1, {\bf x}_2$, and hermitian matrices $M_1\neq M_2$ of rank $n-1$  such that $\Phi(N+\lambda{\bf y}_i{\bf y}_i^{\ast})=M_i+f_i(\lambda) {\bf x}_i{\bf x}_i^{\ast}$ for all $\lambda\neq 0$.
Since ${\cal L}_1\neq {\cal L}_2$, ${\bf y}_1$ and ${\bf y}_2$ are linearly independent. Consequently,
\begin{equation}\label{ii4}
d(B_1,B_2)\geq\rank(B_1-B_2)=2
\end{equation}
for all $B_1\in {\cal L}_1$ and $B_2\in {\cal L}_2$.
Pick $B_1:=N+{\bf y}_1{\bf y}_1^{\ast}$ and set $B_2(\lambda):=N+\lambda{\bf y}_2{\bf y}_2^{\ast}$.
By Corollary~\ref{nova}, matrix  $B_2(\lambda)+{\bf y}_1{\bf y}_1^{\ast}$ is singular for at most one $\lambda$, which we denote it by~$\lambda_0$.
If $\lambda\notin\{0,\lambda_0\}$, then
$\rank\big(B_1-(B_2(\lambda)+{\bf y}_1{\bf y}_1^{\ast})\big)=1=\rank\big((B_2(\lambda)+{\bf y}_1{\bf y}_1^{\ast})-B_2(\lambda)\big)$ so $d(B_1,B_2(\lambda))=2$.
Hence, $d\big(M_1+f_1(1){\bf x}_1{\bf x}_1^{\ast},M_2+f_2(\lambda){\bf x}_2{\bf x}_2^{\ast}\big)=2$ for $\lambda\notin\{0,\lambda_0\}$. Since $q\geq 4$ and $\Phi$ preserves adjacency in both directions,
\begin{equation}\label{eq6}
\rank\big(M_1-M_2+f_1(1){\bf x}_1{\bf x}_1^{\ast}-f_2(\lambda){\bf x}_2{\bf x}_2^{\ast}\big)=2
\end{equation}
holds for at least two distinct nonzero $\lambda_1$ and $\lambda_2$. Let $A$ be the matrix in~(\ref{eq6}) evaluated at $\lambda=\lambda_1$ and pick an invertible $P$ such that $A=P(E_{11}+E_{22})P^{\ast}$. Since $f_2(\lambda_1)-f_2(\lambda_2)\neq 0$ and $\rank\big(A+(f_2(\lambda_1)-f_2(\lambda_2)){\bf x}_2{\bf x}_2^{\ast}\big)=2$, ${\bf x}_2$ is of the form $P(\ast,\ast,0,\ldots,0)^{\tr}P^{\ast}$ by Lemma~\ref{pomozna}. Hence, $M_1-M_2+f_1(1){\bf x}_1{\bf x}_1^{\ast}$ is of the form
\begin{equation}\label{eq7}
P\begin{bmatrix}
*&*&0&\ldots&0\\
*&*&0&\ldots&0\\
0&0&0&\ldots&0\\
\vdots&\vdots&\vdots&\ddots&\vdots\\
0&0&0&\ldots&0
\end{bmatrix}P^{\ast}.
\end{equation}
In a similar way as we obtained~(\ref{eq6}) we deduce that
$$
\rank\big(M_1-M_2+f_1(\lambda){\bf x}_1{\bf x}_1^{\ast}-f_2(\lambda_1){\bf x}_2{\bf x}_2^{\ast}\big)=2
$$
for all nonzero $\lambda$ with one possible exception. Let $\lambda_3\neq 1$ be any such $\lambda$. Since  $\rank\big(A+(f_1(\lambda_3)-f_1(1)){\bf x}_1{\bf x}_1^{\ast}\big)=2$, ${\bf x}_1$ is of the form $P(*,*,0,\ldots,0)^{\tr}P^{\ast}$ by Lemma~\ref{pomozna}.
Consequently, $M_1-M_2$ is of the form~(\ref{eq7}). Therefore,
\begin{equation}\label{ii5}
1\leq \rank(M_1-M_2)\leq 2.
\end{equation}
Moreover, $\rank(M_1-M_2+\mu_1{\bf x}_1{\bf x}_1^{\ast}-\mu_2{\bf x}_2{\bf x}_2^{\ast})\leq 2$ for all $\mu_1,\mu_2\in\FF$. From~(\ref{ii4}), we deduce that $d(\Phi(B_1),\Phi(B_2))\geq 2$, so $\rank\big(\Phi(B_1)-\Phi(B_2)\big)\geq 2$ for all $B_1\in{\cal L}_1$ and $B_2\in{\cal L}_2$.
Hence, $\rank(M_1-M_2+\mu_1{\bf x}_1{\bf x}_1^{\ast}-\mu_2{\bf x}_2{\bf x}_2^{\ast})=2$ for all nonzero $\mu_1,\mu_2$.

In the sequel we may assume that $P$ is the identity matrix. Let $\dot{M}_i$ be the  upper-left $2\times 2$ block of $M_i$, so $M_1-M_2=(\dot{M}_1-\dot{M}_2)\oplus 0_{n-2}$. Similarly, let $\dot{{\bf x}}_i\in\FFq2^2$ be such that ${\bf x}_i=(\dot{{\bf x}}_i^{\tr},0,\ldots,0)^{\tr}$.
Then
\begin{equation}\label{ii6}
\rank(M_1-M_2)=\rank(\dot{M}_1-\dot{M}_2)
\end{equation}
and, for nonzero $\mu_1,\mu_2$, we have
\begin{equation}\label{eq8}
\rank\big(\dot{M}_1-\dot{M}_2+\mu_1\dot{{\bf x}}_1\dot{{\bf x}}_1^{\ast}-\mu_2\dot{{\bf x}}_2\dot{{\bf x}}_2^{\ast}\big)=2.
\end{equation}

Assume that $\rank(\dot{M}_1-\dot{M}_2)=1$, i.e.,  $\dot{M}_1-\dot{M}_2={\bf z}{\bf z}^{\ast}$ for some ${\bf z}\in \FFq2^2$. Then each of the three pairs $\{{\bf z}, \dot{{\bf x}}_1\}$, $\{{\bf z}, \dot{{\bf x}}_2\}$, $\{\dot{{\bf x}}_1, \dot{{\bf x}}_2\}$ consists of two linearly independent vectors. In fact, if for example ${\bf z}$ and $\dot{{\bf x}}_1$ are linearly dependent, then $\mu_1$ can be chosen in such way  that  ${\bf z}{\bf z}^{\ast}+\mu_1\dot{{\bf x}}_1\dot{{\bf x}}_1^{\ast}=0$, which contradicts~(\ref{eq8}). Hence, if $2\times 2$ invertible matrix $Q$ is such that $\dot{{\bf x}}_1=Q(1,0)^{\tr}$ and $\dot{{\bf x}}_2=Q(0,1)^{\tr}$, then ${\bf z}=Q(w_1,w_2)^{\tr}$ for some nonzero $w_1,w_2$. Consequently, for $\mu_2\notin \{0,w_2\inv{w}_2\}$ and $\mu_1:=\mu_2 w_1\inv{w}_1(w_2\inv{w}_2-\mu_2)^{-1}$, the determinant of matrix~(\ref{eq8}) equals
$$\det Q(-\mu_2 w_1\inv{w}_1+\mu_1 w_2\inv{w}_2-\mu_1\mu_2)\det (Q^{\ast})=0,$$ a contradiction. From (\ref{ii5}) and (\ref{ii6}) it follows that $\rank(M_1-M_2)=2$.

Assume that ${\bf x}_1$, ${\bf x}_2$ are linearly independent. Choose $Q$ as above and write $$Q^{-1}(\dot{M}_1-\dot{M}_2)(Q^{-1})^{\ast}=\begin{bmatrix}
\gamma& a\\
\inv{a}&\delta
\end{bmatrix}.$$
From the determinant of matrix (\ref{eq8}) we deduce that
\begin{equation}\label{eq2}
\mu_1(\delta-\mu_2)\neq \gamma \mu_2+a\inv{a}-\gamma\delta
\end{equation}
for all nonzero $\mu_1, \mu_2$. By~(\ref{ii6}), we have $\rank(\dot{M}_1-\dot{M}_2)=2$, so $a\inv{a}-\gamma\delta\neq 0$. Consequently, since $q\geq 4$, there is $\mu_2\in \FF\backslash\{0,\delta\}$ such that $\gamma \mu_2+a\inv{a}-\gamma\delta\neq 0$.
If we choose $\mu_1:=(\gamma \mu_2+a\inv{a}-\gamma\delta)(\delta-\mu_2)^{-1}$ we get in contradiction with~(\ref{eq2}).

Hence, ${\bf x}_1$ and ${\bf x}_2$ are linearly dependent, that is, ${\bf x}_2=c{\bf x}_1$ for some $c\neq 0$.
Choose any invertible $2\times 2$ matrix $R$ with $\dot{\bf x}_1$ as the first column and write
$$R^{-1}(\dot{M}_1-\dot{M}_2)(R^{-1})^{\ast}=\begin{bmatrix}
\varepsilon&b\\
\inv{b} & \eta
\end{bmatrix}.$$ Then, $b\inv{b}-\varepsilon\eta\neq 0$ and~(\ref{eq8}) implies that
\begin{equation*}
\det\begin{bmatrix}
\varepsilon+\mu_1-\mu_2 c\inv{c}&b\\
\inv{b} &\eta
\end{bmatrix}\neq 0
\end{equation*}
for all nonzero $\mu_1$ and $\mu_2$.
Consequently,
\begin{equation}\label{eq3}
(\mu_1-\mu_2c\inv{c})\eta\neq b\inv{b}-\varepsilon\eta.
\end{equation}
If $\eta\neq 0$,  then we get in contradiction with~(\ref{eq3}) for $\mu_1:=\frac{b\inv{b}-\varepsilon\eta}{\eta(1-\zeta)}$ and $\mu_2:=\frac{\zeta\mu_1}{c\inv{c}}$, where $\zeta\in\FF\backslash\{0,1\}$ is arbitrary. Hence $\eta=0$. Consequently,
$$(R^{-1}\oplus I_{n-2})M_i(R^{-1}\oplus I_{n-2})^{\ast}=\begin{bmatrix}
\nu_i&{\bf z}_i^{\ast}\\
{\bf z}_i & D
\end{bmatrix}\qquad (i=1,2)$$
for some $\nu_i$, ${\bf z}_i$, and $D\in\hminus1$ which is the same for $i=1, 2$.
Moreover,
$$(R^{-1}\oplus I_{n-2})(M_i+\lambda {\bf x}_1{\bf x}_1^{\ast})(R^{-1}\oplus I_{n-2})^{\ast}=\begin{bmatrix}
\nu_i+\lambda&{\bf z}_i^{\ast}\\
{\bf z}_i & D
\end{bmatrix}\qquad (i=1,2)$$
is invertible if and only if $\lambda\neq 0$. Hence, a Laplace decomposition of the determinant on the first row shows that $\det D\neq 0$, that is, $D$ is invertible. By Lemma~\ref{lemma6} (ii), $\nu_i={\bf z}_i^{\ast}D^{-1}{\bf z}_i$ and for $\lambda\neq 0$ we have
$$(R\oplus I_{n-2})^{\ast}(M_i+\lambda {\bf x}_1{\bf x}_1^{\ast})^{-1}(R\oplus I_{n-2})=
\begin{bmatrix}
0&0\\
0&D^{-1}
\end{bmatrix}+\lambda^{-1}
\begin{bmatrix}
-1\\
D^{-1}{\bf z}_i
\end{bmatrix}
\begin{bmatrix}
-1\\
D^{-1}{\bf z}_i
\end{bmatrix}^{\ast}.$$
Hence
$\Phi({\cal L}_i)^{-1}=\{(M_i+\lambda {\bf x}_i{\bf x}_i^{\ast})^{-1}\ :\ 0\neq \lambda\in \FF\}$ is a leaf of matrix
\begin{equation}\label{eq9}
(R^{-1}\oplus I_{n-2})^{\ast}\begin{bmatrix}
0&0\\
0&D^{-1}
\end{bmatrix}(R^{-1}\oplus I_{n-2}).
\end{equation}
This completes the proof since matrix~(\ref{eq9}) does not depend on $i=1,2$.
\end{proof}

In the final part of the proof of Theorem~\ref{thm1} we will try to extend adjacency preserving map $\Phi$ from the set $\hgl$ to $\h$. The next lemma will provide us some information on the matrix of a leaf $\Phi({\cal L})$.

\begin{lemma}\label{lemma10}
Assume $q\geq 3$, $M_1\neq M_2$ are hermitian matrices with $\rank M_1=r$, $\rank M_2\in\{r-1,r\}$, where $1\leq r\leq n-1$, and  ${\cal L}_i=\{M_i+\lambda {\bf x}_i{\bf x}_i^{\ast}\ :\ 0\neq \lambda\in \FF\}$ is a leaf of $M_i$.
Then for any $\lambda_1\in \FF\backslash\{0\}$ there is a unique  $\lambda_2\in \FF\backslash\{0\}$, and for any $\lambda_2\in \FF\backslash\{0\}$ there is a unique $\lambda_1\in \FF\backslash\{0\}$ such that
\begin{equation}\label{eq10}
\rank\big((M_1+\lambda_1 {\bf x}_1{\bf x}_1^{\ast})-(M_2+\lambda_2 {\bf x}_2{\bf x}_2^{\ast})\big)=1
\end{equation}
if and only if one of the following holds:
\begin{enumerate}
\item ${\bf x}_1, {\bf x}_2$ are linearly independent, $M_1-M_2=a{\bf x}_1{\bf x}_2^{\ast}+\inv{a}{\bf x}_2{\bf x}_1^{\ast}$ for some $a\neq 0$,
\item ${\bf x}_1, {\bf x}_2$ are linearly dependent and $\rank(M_1-M_2)=1$.
\end{enumerate}
\end{lemma}
\begin{proof}
We first prove the sufficient part. In case (i) we choose an invertible matrix~$P$ with ${\bf x}_i$ as the $i$-th column. The matrix in~(\ref{eq10}) can be rewritten as
$$P\left(\begin{bmatrix}\lambda_1& a\\
\inv{a}& -\lambda_2\end{bmatrix}\oplus 0_{n-2}\right)P^{\ast}.$$ Given $\lambda_1\neq 0$, only $\lambda_2:=-a\inv{a}\lambda_1^{-1}$ fits~(\ref{eq10}). Similarly, given $\lambda_2\neq 0$, only $\lambda_1:=-a\inv{a}\lambda_2^{-1}$ fits~(\ref{eq10}).
If (ii) holds then ${\bf x}_2=b{\bf x}_1$ for some $b\neq 0$. Given $\lambda_1\neq 0$, $\lambda_2:=(b\inv{b})^{-1}\lambda_1$ fits~(\ref{eq10}). If $\lambda_2\neq (b\inv{b})^{-1}\lambda_1$, then~(\ref{eq10}) would imply that $\rank(M_1-M_2+\mu {\bf x}_1{\bf x}_1^{\ast})=1$ for some $\mu\neq 0$. Since $\rank(M_1-M_2)=1$, we deduce that $M_1-M_2=\nu {\bf x}_1{\bf x}_1^{\ast}$ for some $\nu\neq 0$, that is,  $M_2=M_1-\nu {\bf x}_1{\bf x}_1^{\ast}\in {\cal L}_1$. Consequently $\rank M_2=r+1$, a contradiction. In the same way we see that given $\lambda_2\neq 0$, there is a unique $\lambda_1\neq 0$ that fits ~(\ref{eq10}).

We now prove the necessity of condition (i) or (ii). Assume first that ${\bf x}_1, {\bf x}_2$ are linearly independent. Pick an invertible matrix $P$ as above.  By assumption~(\ref{eq10}), for any $\lambda_1\in \FF\backslash\{0\}$ there is a unique $\lambda_2\in \FF\backslash\{0\}$ such that
$N:=P^{-1}(M_1-M_2)(P^{-1})^{\ast}+\lambda_1E_{11}-\lambda_2E_{22}$ is of rank one. Let $a_{ij}$ denote $(i,j)$-th entry of $P^{-1}(M_1-M_2)(P^{-1})^{\ast}$. If $a_{1j}\neq 0$ for some $j\geq 3$, then the $2\times 2$ minor of $N$ formed by first and $j$-th row/column vanishes for at most one $\lambda_1$, which contradicts the assumption. Hence, $a_{1j}=0$ for all $j\geq 3$. Similarly we see that $a_{2j}=0$ for all $j\geq 3$. It is now obvious that  $a_{ij}=0$ if $i\geq 3$ or $j\geq 3$. If $a_{11}\neq 0$ then we choose $\lambda_1:=-a_{11}$. The  upper-left $2\times 2$ minor of $N$  vanishes only if $a_{12}=0$, however in that case more than one $\lambda_2$ posses the required property. Hence, $a_{11}=0$. Similarly we see that $a_{22}=0$. Consequently, $M_1-M_2=a_{12}{\bf x}_1{\bf x}_2^{\ast}+\inv{a}_{12}{\bf x}_2{\bf x}_1^{\ast}$ and $a_{12}\neq 0$, since $M_1\neq M_2$.

Assume now that ${\bf x}_1, {\bf x}_2$ are linearly dependent, that is, ${\bf x}_2=b{\bf x}_1$ for some $b\neq 0$.
Pick an invertible matrix $Q$ with ${\bf x}_1$ as its first column and let $g:\FF\backslash \{0\}\to \FF\backslash \{0\}$ be the bijection $\lambda_1\mapsto \lambda_2$ described in~(\ref{eq10}).
Then $N:=Q^{-1}(M_1-M_2)(Q^{-1})^{\ast}+(\lambda_1-g(\lambda_1)b\inv{b})E_{11}$ is of rank one for all $\lambda_1\in\FF\backslash\{0\}$. Let $a_{ij}$ denote $(i,j)$-th entry of $Q^{-1}(M_1-M_2)(Q^{-1})^{\ast}$.

If $\lambda_1\mapsto \lambda_1-g(\lambda_1)b\inv{b}$ is non-constant, then $a_{1j}=0$ for $j\geq 2$ since otherwise,  the $2\times 2$ minor of $N$ formed by first and $j$-th row/column would not vanish for all values $\lambda_1-g(\lambda_1)b\inv{b}$. Consequently, $a_{ij}=0$ if $i\geq 2$ or $j\geq 2$, that is, $M_1-M_2=a_{11}{\bf x}_1{\bf x}_1^{\ast}$. Hence, $M_2=M_1-a_{11}{\bf x}_1{\bf x}_1^{\ast}\in {\cal L}_1\cup\{M_1\}$, a contradiction.

We have proved that  $\lambda_1-g(\lambda_1)b\inv{b}\equiv \gamma$ for some constant $\gamma$. If $\gamma\neq 0$, then $g(\gamma)=0$, a contradiction. Hence, $\lambda_1-g(\lambda_1)b\inv{b}\equiv 0$ and $\rank (M_1-M_2)=1$.
\end{proof}

The next two lemmas describe connected components of the subgraph in $\h$, which is induced by matrices of rank $n-1$.

\begin{lemma}\label{pomozna2}
Let $N_1,N_2\in \h$ be adjacent and of rank $n-1$. If $\{N_1+\lambda {\bf y}{\bf y}^{\ast}\ : \ 0\neq\lambda\in\FF\}$ is a leaf of $N_1$, then $\{N_2+\lambda {\bf y}{\bf y}^{\ast} : 0\neq\lambda\in\FF\}$ is a leaf of~$N_2$.
\end{lemma}
\begin{proof}
By (\ref{i8}), $N_1=P(I_{n-1}\oplus 0)P^{\ast}$ for some invertible  matrix $P$. Let $P^{-1}{\bf y}=:\dot{\bf y}=(\dot{y}_1,\ldots,\dot{y}_{n})^{\tr}$. Since $N_1+\lambda {\bf y}{\bf y}^{\ast}$ is invertible for $\lambda\neq 0$, it follows that $\dot{y}_{n}\neq 0$. Since $N_1$ and $N_2$ are adjacent, $N_2=N_1+{\bf x}{\bf x}^{\ast}$ for some ${\bf x}$. Let  $P^{-1}{\bf x}=:\dot{\bf x}=(\dot{x}_1,\ldots,\dot{x}_{n})^{\tr}$. Since $\rank N_2=n-1$, we deduce that $\dot{x}_{n}=0$ by Lemma~\ref{pomozna}. Moreover, $N_2+\lambda {\bf y}{\bf y}^{\ast}$ is invertible for $\lambda\neq 0$.
\end{proof}

\begin{lemma}\label{lemma11}
Let $q\geq 3$. The subgraph in $\her$, induced by matrices of rank $n-1$, has $\frac{q^{2n}-1}{q^2-1}$ connected components, which are precisely the sets
\begin{equation}\label{ii9}
\{P(\dot{X}\oplus 0)P^{\ast} :  \dot{X}\in\hglnminus1 \},
\end{equation}
where $P$ is invertible, so any component is isomorphic to graph $\hglnminus1$.
\end{lemma}
\begin{proof}
Let $A\in\h$ be of rank $n-1$. Then there exist an invertible matrix~$P$ and a matrix $\dot{A}\in \hglnminus1$ such that $A=P(\dot{A}\oplus 0)P^{\ast}$. Assume that $B$ is of rank $n-1$ and it is adjacent to $A$, that is, there exists a vector ${\bf x}$ such that
\begin{equation}\label{ii7}
B=A+{\bf x}{\bf x}^{\ast}=P(\dot{A}\oplus 0+{\bf w}{\bf w}^{\ast})P^{\ast},
\end{equation}
where $P^{-1}{\bf x}=:{\bf w}=(w_1,\ldots,w_n)^{\tr}$.
Since $\rank B=n-1$, Lemma~\ref{pomozna} shows that $w_n=0$, so (\ref{ii7}) implies that
$B=P(\dot{B}\oplus 0)P^{\ast}$ for some $\dot{B}\in\hglnminus1$. Hence, the component which contains $A$ contains at most those matrices from the set~(\ref{ii9}). However, $\hglnminus1$ is connected  by Lemma~\ref{lemma_povezan}, so the component of $A$ is truly the whole set~(\ref{ii9}), which equals
$$\left\{\sum_{i=1}^{n-1} {\bf  x}_i{\bf x}_i^{\ast}\ :\ \{{\bf  x}_1,\ldots,{\bf  x}_{n-1}\}\
\textrm{is a basis of}\ V\right\},$$
where $V$ is the linear span of vectors $P{\bf  e}_1,\ldots, P{\bf  e}_{n-1}$.
Consequently, the map
$$V\mapsto \left\{\sum_{i=1}^{n-1} {\bf  x}_i{\bf x}_i^{\ast}\ :\ \{{\bf  x}_1,\ldots,{\bf  x}_{n-1}\}\
\textrm{is a basis of}\ V\right\}$$
is a bijection between the set of all $n-1$ dimensional subspaces in $\FF_{q^{2}}^n$ and the set of all components.
Hence, the number of all components equals the number of all $n-1$ dimensional subspaces, which is given by Gaussian binomial coefficient
$$\begin{bmatrix}n\\
n-1\end{bmatrix}_{q^{2}}=\frac{(q^2)^n-1}{q^2-1}.$$
An isomorphism from $\hglnminus1$ to component~(\ref{ii9}) is  $\dot{X}\mapsto P(\dot{X}\oplus 0)P^{\ast}$.
\end{proof}
\begin{remark}
Similar proof as above shows that any connected component of the subgraph in $\h$, which is induced by matrices of rank $k$ $(1\leq k\leq n-2)$, is of the form
$\{P(\dot{X}\oplus 0_{n-k})P^{\ast}\ :\  \dot{X}\in\mathcal{HGL}_{k}(\FFq2)\}$. There are  $\left[\begin{smallmatrix}n\\
k\end{smallmatrix}\right]_{q^{2}}$ such components.
\end{remark}
\begin{remark}
The assumption $q\geq 3$ in Lemma~\ref{lemma11} is needed just because Lemma~\ref{lemma_povezan} is used in the proof. In \cite{petersen_coxeter} the analogous result of Lemma~\ref{lemma_povezan} is proved for $q=2$, so the conclusion of Lemma~\ref{lemma11} is valid for any $q$.
\end{remark}

Let $q\geq 3$, $1\leq r\leq n$, and assume $U\subseteq \h$ is a subset of all matrices of rank~$r$. Let
$f$ be a flower of some matrix of rank $r-1$ and let $\Phi: U \to U$ be a bijection that preserves adjacency in both directions. We say that $\Phi$ \emph{preserves flower $f$} if $\Phi(f)$ is a flower of some matrix of rank $r-1$. We say that $\Phi$ \emph{disintegrates flower $f$} if, for any two  leaves ${\cal L}_i\neq {\cal L}_j$ of $f$, $\Phi({\cal L}_i)$ and $\Phi({\cal L}_j)$ are leaves in two distinct flowers of matrices of rank $r-1$. Given a connected component ${\cal C}$ of the subgraph of $\h$, which is induced by matrices of rank $n-1$, we say that $\Phi$ \emph{preserves ${\cal C}$}, if it preserves all flowers of matrices in ${\cal C}$.
We say that $\Phi$ \emph{disintegrates~${\cal C}$}, if it disintegrates all flowers  of matrices in ${\cal C}$.

\begin{lemma}\label{lemma44}
Let $q\geq 4$. A bijection $\Phi: \hgl \to\hgl$ that preserve adjacency in both directions either preserves all flowers of matrices of rank $n-1$ or disintegrates all flowers of matrices of rank $n-1$.
\end{lemma}
\begin{proof} We split the proof in three steps.
\begin{myenumerate}{Step}
\item $\Phi$ either preserves or disintegrates any flower of a matrix of rank $n-1$.

Suppose that the claim  is false for some flower $f$. Then there exist distinct leaves ${\cal L}_1$, ${\cal L}_2$, ${\cal L}_3$ of $f$ and distinct  matrices $M$, $N$ of rank $n-1$ such that
$\Phi({\cal L}_1)$ and $\Phi({\cal L}_2)$ are both leaves of  $M$, while $\Phi({\cal L}_3)$ is a leaf of  $N$.
By Lemma~\ref{lemma7}, $\Phi({\cal L}_1)^{-1}$ and $\Phi({\cal L}_3)^{-1}$ are leaves of a common matrix of rank $n-1$. The same is true for leaves $\Phi({\cal L}_2)^{-1}$ and $\Phi({\cal L}_3)^{-1}$. This is a contradiction since, by Corollary~\ref{lemma22},  $\Phi({\cal L}_1)^{-1}$ and $\Phi({\cal L}_2)^{-1}$ are leaves of two distinct matrices of rank $n-1$.

\item  Let ${\cal C}$ be a connected component of the subgraph in $\h$, which is induced by matrices of rank $n-1$. Then $\Phi$  either preserves or disintegrates~${\cal C}$.

Suppose the claim is false. Then there are two flowers $f_1$ and $f_2$ of matrices $N_1$ and $N_2$ respectively such that $ N_i\in {\cal C}$, $\rank(N_1-N_2)=1$, and $\Phi$ preserves $f_1$ while it disintegrates $f_2$. From Section~\ref{2} we know that flower $f_i$ consists of $q^{2(n-1)}$ leaves ${\cal L}_1^{(i)},\ldots,{\cal L}_{q^{2(n-1)}}^{(i)}$. By Lemma~\ref{lemma7} there are matrices $M_1,\ldots,M_{q^{2(n-1)}}$ of rank~$n-1$ and a column vector ${\bf x}\neq 0$ such that $\Phi({\cal L}_j^{(2)})=\{M_j+\lambda {\bf x}{\bf x}^{\ast} : 0\neq \lambda\in \FF\}$ and $\rank(M_j-M_k)=2$ for $j\neq k$.
Let $M$, with $\rank M=n-1$, be the matrix of flower~$\Phi(f_1)$. Since $\Phi$ is bijective, $\Phi({\cal L}_j^{(2)})$ is not a leaf in $\Phi(f_1)$, so $M\neq M_j$ for all $j$.

We claim that $\rank(M-M_j)\neq 1$ for at least some $j$. The opposite would imply that $M_j=M+{\bf x}_j{\bf x}_j^{\ast}$ for some ${\bf x}_j\in \FFq2^n$ for all $j$. Since $\rank(M_j-M_k)=2$ for $j\neq k$, column vectors ${\bf x}_1,\ldots,{\bf x}_{q^{2(n-1)}}$ must be pairwise linearly independent. Moreover, none of the sets $\{M+\lambda {\bf x}_j{\bf x}_j^{\ast}\ :\ 0\neq\lambda\in\FF\}$ is a leaf of $M$, since $\rank M_j<n$. On the contrary, $\FFq2^n$ contains $\frac{q^{2n}-1}{q^2-1}$ pairwise linearly independent column vectors. By excluding those $q^{2(n-1)}$ column vectors that generate leaves of $M$, we obtain a number $\frac{q^{2n}-1}{q^{2}-1}-q^{2(n-1)}=\frac{q^{2(n-1)}-1}{q^{2}-1}$, which is smaller than $q^{2(n-1)}=|\{{\bf x}_1,\ldots,{\bf x}_{q^{2(n-1)}}\}|$,
a contradiction.

Now pick $M_j$ such that $\rank (M-M_j)\neq 1$ and choose ${\bf y}$ such that ${\cal L}_j^{(2)}=\{N_2+\lambda {\bf y}{\bf y}^{\ast}\ :\ 0\neq \lambda\in \FF\}$. By Lemma~\ref{pomozna2}, ${\cal L}:=\{N_1+\lambda {\bf y}{\bf y}^{\ast}\ :\ 0\neq \lambda\in \FF\}$ is a leaf in~$f_1$. Pick ${\bf z}$ such that $\Phi({\cal L})=\{M+\lambda {\bf z}{\bf z}^{\ast}\ :\ 0\neq \lambda\in \FF\}$. Since ${\cal L}$ and ${\cal L}_j^{(2)}$ satisfies condition~(\ref{eq10}), the same holds for $\Phi({\cal L})$ and $\Phi({\cal L}_j^{(2)})$. Since $\rank (M-M_j)\neq 1$, Lemma~\ref{lemma10} implies that $M_j-M=a{\bf x}{\bf z}^{\ast}+\inv{a}{\bf z}{\bf x}^{\ast}$ for some nonzero $a$. To continue write $M=\sum_{i=1}^{n-1} {\bf z}_i{\bf z_i}^{\ast}$ and choose the invertible matrix $P$ with ${\bf z}_i$ as the $i$-th column and ${\bf z}$ as the last column. If ${\bf w}:=P^{-1}{\bf x}$ then
$$M_j=M+a{\bf x}{\bf z}^{\ast}+\inv{a}{\bf z}{\bf x}^{\ast}=P
\begin{bmatrix}
1& & &aw_1\\
&\ddots& &\vdots\\
& & 1& aw_{n-1}\\
\inv{aw_1}&\cdots& \inv{aw_{n-1}}&aw_n+\inv{aw_n}
\end{bmatrix}P^{\ast},$$
where ${\bf w}=:(w_1\ldots,w_{n})^{\tr}$. Since $P^{-1}M_j (P^{-1})^{\ast}$ is of rank $n-1$, with first $n-1$ columns linearly independent, it follows that the last column of this matrix is a linear combination of others, that is, $aw_n+\inv{aw_n}=\sum_{i=1}^{n-1}aw_i\inv{aw_i}$. Hence,
$$M_j=P\begin{bmatrix}
I_{n-1}& {\bf u}\\
{\bf u}^{\ast}& {\bf u}^{\ast}{\bf u}
\end{bmatrix}P^{\ast}$$
for ${\bf u}:=(aw_1,\ldots,aw_{n-1})^{\tr}$, while
$$M=P\begin{bmatrix}
I_{n-1}& 0\\
0& 0
\end{bmatrix}P^{\ast}.$$
By Corollary~\ref{cor2} there are two leaves, one of $M_j$ and one of $M$, with inverses that share a common matrix of rank $n-1$. More precisely, there exist a leaf ${\cal L}_i^{(1)}$ in $f_1$, a flower $f_3$ of some matrix $N_3$ of rank $n-1$ with leaves ${\cal L}_1^{(3)},\ldots,{\cal L}_{q^{2(n-1)}}^{(3)}$, and index $k$ such that $\Phi({\cal L}_k^{(3)})$ is a leaf of $M_j$, while $\Phi\big({\cal L}_k^{(3)}\big)^{-1}$ and $\Phi\big({\cal L}_i^{(1)}\big)^{-1}$ are leaves in the same flower. Since $\Phi({\cal L}_k^{(3)})$ and $\Phi({\cal L}_j^{(2)})$ are leaves in the same flower and $\Phi$ is bijective, $f_3$ is as $f_2$ disintegrated by $\Phi$. Hence, $N_3\neq N_1$ and
$\{\Phi\big({\cal L}_1^{(3)}\big)^{-1},\ldots,\Phi\big({\cal L}_{q^{2(n-1)}}^{(3)}\big)^{-1}\}$ is a flower by Lemma~\ref{lemma7}. Since $\Phi\big({\cal L}_i^{(1)}\big)^{-1}$ is an additional leaf of this flower, we get a contradiction.

\item We are now able to end the proof.

Assume that $\Phi$ disintegrates some flower of a matrix of rank $n-1$. Then, by Step~1, the same holds for inverse $\Phi^{-1}$. In fact, if ${\cal L}$ is a leaf of $f$, which is disintegrated by $\Phi$, then $\Phi({\cal L})$ is a leaf of a flower, denoted by $g$, which is disintegrated by $\Phi^{-1}$. Let  ${\cal L}_1,\ldots,{\cal L}_{q^{2(n-1)}}$ be the leaves of $g$ and let $M_i$  be the matrix of rank~$n-1$ of leaf $\Phi^{-1}({\cal L}_i)$.

We claim that matrices $M_1,\ldots, M_{q^{2(n-1)}}$ are in different connected component of the subgraph in $\h$, which is induced by matrices of rank $n-1$. Let $i\neq j$. By Lemma~\ref{lemma7}, $(\Phi^{-1}({\cal L}_i))^{-1}$ and $(\Phi^{-1}({\cal L}_j))^{-1}$ are leaves of a common matrix of rank $n-1$. By Corollary~\ref{cor2},
$$M_i=P\begin{bmatrix}
I_{n-1}& {\bf y}_i\\
{\bf y}_i^{\ast}& {\bf y}_i^{\ast}{\bf y}_i
\end{bmatrix}P^{\ast}\quad \textrm{and} \quad M_j=P\begin{bmatrix}
I_{n-1}& {\bf y}_j\\
{\bf y}_j^{\ast}& {\bf y}_j^{\ast}{\bf y}_j
\end{bmatrix}P^{\ast}.$$
for some invertible $P$ and column vectors ${\bf y}_i\neq {\bf y}_j$.
 Consequently, $$M_i=Q\begin{bmatrix}
I_{n-1}& 0\\
0& 0
\end{bmatrix}Q^{\ast}\quad \textrm{and} \quad M_j=Q\begin{bmatrix}
I_{n-1}& {\bf y}_j-{\bf y}_i\\
{\bf y}_j^{\ast}-{\bf y}_i^{\ast}& ({\bf y}_j-{\bf y}_i)^{\ast}({\bf y}_j-{\bf y}_i)
\end{bmatrix}Q^{\ast},$$
where
$$Q:=P\begin{bmatrix}
I_{n-1}& 0\\
{\bf y}_i^{\ast}& 1
\end{bmatrix}.$$
Since ${\bf y}_j-{\bf y}_i\neq 0$, $M_i$ and $M_j$ are in distinct components by Lemma~\ref{lemma11}.

Consequently, since $\Phi$ disintegrates flowers of all $M_1,\ldots,M_{q^{2(n-1)}}$, Step~2 implies that $\Phi$ disintegrates at least $q^{2(n-1)}$ components. By Lemma~\ref{lemma11} there are
$\frac{q^{2n}-1}{q^2-1}< 2\cdot q^{2(n-1)}$  components in total. So if ${\mathfrak d}_\Phi$ denotes the number of components that are disintegrated by~$\Phi$ and ${\mathfrak p}_\Phi$ denotes the number of components that are preserved by~$\Phi$, then ${\mathfrak d}_\Phi > {\mathfrak p}_\Phi$. Let $\Upsilon(A):=A^{-1}$. Since $\Upsilon$ is bijective and preserves adjacency in both directions, Corollary~\ref{lemma22} and Lemma~\ref{lemma7} imply that ${\mathfrak p}_\Phi={\mathfrak d}_{\Upsilon\circ \Phi}$ and ${\mathfrak d}_\Phi={\mathfrak p}_{\Upsilon\circ \Phi}$. If $\Phi$ preserves some flower, that is, if ${\mathfrak p}_\Phi={\mathfrak d}_{\Upsilon\circ \Phi}>0$, then a symmetrical argument used for the map $\Upsilon\circ \Phi$ shows that ${\mathfrak d}_{\Upsilon\circ \Phi} > {\mathfrak p}_{\Upsilon\circ \Phi}$. Hence, ${\mathfrak p}_\Phi > {\mathfrak d}_\Phi$, a contradiction. Therefore $\Phi$ disintegrates all flowers.\qedhere
\end{myenumerate}
\end{proof}

We are now able to prove Theorem~\ref{thm1}.

\begin{proof}[Proof of Theorem~\ref{thm1}]
It is obvious that  maps in~(\ref{ii10}) are bijective and preserve adjacency in both directions.
We now prove that no other exists. Assume that $\Phi: \hgl\to\hgl$ is bijective and preserves adjacency in both directions. Denote the map $A\mapsto A^{-1}$ by $\Upsilon$.
By Lemma~\ref{lemma44}, $\Phi$ either preserves all flowers or disintegrates all flowers of matrices of rank $n-1$. In the first case we define $\Psi:=\Phi$, while in the second case we set $\Psi:=\Upsilon\circ \Phi$. An application of Lemma~\ref{lemma7} shows that $\Psi$ preserves all flowers of matrices of rank $n-1$. In the next three steps we show that  $\Psi$ can be bijectively extended to a map ${\cal H}_n(\mathbb{F}_{q^2})\to {\cal H}_n(\mathbb{F}_{q^2})$ that preserves adjacency (in both directions).
\begin{myenumerate}{Step}
\item We first extend $\Psi$ bijectively on matrices of rank $n-1$ in such way that $\Psi$ obeys the rule
\begin{equation}\label{eq12}
\rank(M_1-M_2)=1, \rank M_1 =n, \rank M_2=n-1\Longrightarrow \rank\big(\Psi(M_1)-\Psi(M_2)\big)=1.
\end{equation}

If $M\in {\cal H}_n(\mathbb{F}_{q^2})$ is of rank $n-1$, $f$ is its flower, and $N$ is the matrix of rank~$n-1$ of the flower $\Psi(f)$, then we set $\Psi(M):=N$. By the construction, $\Psi$ satisfies~(\ref{eq12}) and is bijective on matrices of rank $\geq n-1$.

\item If $1\leq r\leq n-1$, $\Psi$ is defined on matrices of rank $\geq r$, it preserves adjacency, and for all $r \leq k\leq n-1$ is bijective on the subset of matrices of rank~$k$, then we extend $\Psi$ bijectively on matrices of rank $r-1$ such that $\Psi$ satisfies
\begin{equation*}\label{eq13}
\rank(M_1-M_2)=1, \rank M_1 =r, \rank M_2=r-1\Longrightarrow \rank\big(\Psi(M_1)-\Psi(M_2)\big)=1.
\end{equation*}

It suffices to show that $\Psi$ preserves flowers of matrices of rank $r-1$. Then we can define $\Psi$ on matrices of rank $r-1$ in analogous way as in Step~1.

Let $A$ be of rank $r-1$ and let $\{A+\lambda {\bf x}{\bf x}^{\ast}\, :\, 0\neq \lambda\in\FF\}$, $\{A+\lambda {\bf y}{\bf y}^{\ast}\, :\, 0\neq \lambda\in\FF\}$ be two leaves of $A$. By Lemma~\ref{list_v_list}, their $\Psi$-images are leaves of some matrices of rank $r-1$, denoted by $M_{\bf x}$ and $M_{\bf y}$ respectively. We need to show that $M_{\bf x}=M_{\bf y}$. Write $A=\sum_{i=1}^{r-1} {\bf u}_i{\bf u}_i^{\ast}$, where ${\bf u}_1\ldots,{\bf u}_{r-1}$ are linearly independent.

We first assume that ${\bf u}_1\ldots,{\bf u}_{r-1}, {\bf x}, {\bf y}$ are linearly independent.
Then $\rank(A+\lambda {\bf x}{\bf x}^{\ast}+\mu {\bf y}{\bf y}^{\ast})=r+1$ for all nonzero $\lambda$ and $\mu$.   For $\lambda\neq 0$ let $M_{\lambda}:=\Psi(A+\lambda {\bf x}{\bf x}^{\ast})$. Leaf $\{A+\nu {\bf y}{\bf y}^{\ast}\, :\, 0\neq \nu\in\FF\}$ of $A$ and leaf $\{A+\lambda {\bf x}{\bf x}^{\ast}+\nu {\bf y}{\bf y}^{\ast}\, :\, 0\neq \nu\in\FF\}$ of $A+\lambda {\bf x}{\bf x}^{\ast}$ satisfy condition (\ref{eq10}) in Lemma~\ref{lemma10}. The same holds for their $\Psi$-images, which are of the form $\{M_{\bf y}+\nu {\bf z}{\bf z}^{\ast}\, :\, 0\neq \nu\in\FF\}$ and $\{M_{\lambda}+\nu {\bf x}_{\lambda}{\bf x}_{\lambda}^{\ast}\, :\, 0\neq \nu\in\FF\}$ respectively. Hence, Lemma~\ref{lemma10} implies that  $\rank(M_{\bf y}-M_{\lambda})\in\{1,2\}$. Since $q\geq 4$, it suffices to consider the next two cases.
\begin{myenumerate2}{Case}
\item There exist $\lambda\neq \mu$ such that $\rank(M_{\bf y}-M_{\lambda})=2=\rank(M_{\bf y}-M_{\mu})$.

Then, by Lemma~\ref{lemma10},
$M_{\bf y}-M_{\lambda}=a{\bf z}{\bf x}_{\lambda}^{\ast}+\inv{a}{\bf x}_{\lambda}{\bf z}^{\ast}$ and $M_{\bf y}-M_{\mu}=b{\bf z}{\bf x}_{\mu}^{\ast}+\inv{b}{\bf x}_{\mu}{\bf z}^{\ast}$ for some nonzero $a$ and $b$. Since leaves $\{A+\lambda {\bf x}{\bf x}^{\ast}+\nu {\bf y}{\bf y}^{\ast}\, :\, 0\neq \nu\in\FF\}$ and $\{A+\mu {\bf x}{\bf x}^{\ast}+\nu {\bf y}{\bf y}^{\ast}\, :\, 0\neq \nu\in\FF\}$ satisfy condition (\ref{eq10}), the same holds for their $\Psi$-images. Since $M_{\lambda}$ and $M_{\mu}$ are adjacent, it follows from Lemma~\ref{lemma10} that ${\bf x}_{\lambda}$ and ${\bf x}_{\mu}$ are linearly dependent, that is, ${\bf x}_{\mu}=c{\bf x}_{\lambda}$ for some $c$. Consequently, $M_{\mu}-M_{\lambda}=(a-b\inv{c}){\bf z}{\bf x}_{\lambda}^{\ast}+\inv{(a-b\inv{c})}{\bf x}_{\lambda}{\bf z}^{\ast}$ is of rank two or zero, a contradiction.

\item There exist $\lambda\neq \mu$ such that $\rank(M_{\bf y}-M_{\lambda})=1=\rank(M_{\bf y}-M_{\mu})$.

 Then (\ref{ii1}) implies that $M_{\bf y}$ is adjacent to all matrices in $\Psi(\{A+\nu {\bf x}{\bf x}^{\ast} : 0\neq \nu\in\FF\})$, which, by Lemma~\ref{list_v_list}, is a leaf of unique matrix of rank $r-1$, so
       $M_{\bf y}=M_{\bf x}$.
\end{myenumerate2}

Assume now that ${\bf u}_1\ldots,{\bf u}_{r-1}, {\bf x}, {\bf y}$ are linearly dependent. It is obvious that
${\bf u}_1, \ldots, {\bf u}_{r-1}, {\bf x}$  as well as ${\bf u}_1, \ldots, {\bf u}_{r-1}, {\bf y}$ are linearly independent. So if ${\bf v}$ is such that ${\bf u}_1\ldots,{\bf u}_{r-1}, {\bf x}, {\bf v}$ are linearly independent, then ${\bf u}_1\ldots,{\bf u}_{r-1}, {\bf y}, {\bf v}$ are also linearly independent. Let $M_{\bf v}$ be the matrix of rank $r-1$ of leaf $\Psi(\{A+\lambda {\bf v}{\bf v}^{\ast} : 0\neq \lambda\in\FF\})$. The proof above shows that $M_{\bf x}=M_{\bf v}$ and $M_{\bf y}=M_{\bf v}$, i.e., $M_{\bf x}=M_{\bf y}$.

\item If $2\leq r\leq n$, then $\Psi$ from the conclusion of Step~2/Step~1 satisfies
\begin{equation*}\label{eq14}
\rank(M_1-M_2)=1, \rank M_1 =r-1, \rank M_2=r-1\Longrightarrow \rank\big(\Psi(M_1)-\Psi(M_2)\big)=1.
\end{equation*}

Assume that $\rank(M_1-M_2)=1$, $\rank M_1 =r-1$, $\rank M_2=r-1$. From the construction of $\Psi$ in Step~2/Step~1 we deduce that the $\Psi$-image of any leaf $\{M_i+\lambda {\bf w}{\bf w}^{\ast}\, :\, 0\neq \lambda\in\FF\}$ of $M_i$ is of the form $\{\Psi(M_i)+\lambda h_i({\bf w})h_i({\bf w})^{\ast}\, :\, 0\neq \lambda\in\FF\}$ for some function $h_i$. Moreover, the set of leaves of $M_i$ is bijectively mapped onto the set of leaves of $\Psi(M_i)$. If $\{M_1+\lambda {\bf w}{\bf w}^{\ast}\, :\, 0\neq \lambda\in\FF\}$ is any leaf of $M_1$, then $\{M_2+\lambda {\bf w}{\bf w}^{\ast}\, :\, 0\neq \lambda\in\FF\}$ is a leaf of $M_2$ by Lemma~\ref{pomozna2}. These two leaves satisfy~(\ref{eq10}), so the same hold for their $\Psi$-images. Assume erroneously that  $\rank\big(\Psi(M_1)-\Psi(M_2)\big)\neq 1$. Then  Lemma~\ref{lemma10} implies that
\begin{equation}\label{ii16}
\Psi(M_1)-\Psi(M_2)=a_{{\bf w}} h_1({\bf w})h_2({\bf w})^{\ast}+\inv{a}_{{\bf w}}  h_2({\bf w})h_1({\bf w})^{\ast}
\end{equation}
for some nonzero scalar $a_{{\bf w}}$.

If $n\geq 3$, then it is easy to see that there are three leaves of $\Psi(M_1)$ that are generated by three linearly independent vectors $h_1(\dot{{\bf w}}), h_1(\ddot{{\bf w}}), h_1(\dddot{{\bf w}})$. However, in that case the three matrices (\ref{ii16}), obtained by choosing ${\bf w}\in \{\dot{{\bf w}},\ddot{{\bf w}},\dddot{{\bf w}}\}$, cannot be equal, a contradiction.

If $n=2$, then $r=2$ and $\rank M_i=1=\rank \Psi(M_i)$. In particular, $\Psi(M_i)={\bf v}_i{\bf v}_i^{\ast}$, where  ${\bf v}_1$ and ${\bf v}_2$ are linearly independent. By (\ref{i8}) there is an invertible $2\times 2$ matrix~$Q$ such that $\Psi(M_1)-\Psi(M_2)=QQ^{\ast}$. Let ${\bf e}_1:=(1,0,\ldots,0)^{\tr}$. At least one of the pairs $\{Q{\bf e}_1, {\bf v}_1\}$ and $\{Q{\bf e}_1, {\bf v}_2\}$ consists of two linearly independent vectors. We may assume the former pair is such. Then $\{\Psi(M_1)+\lambda(Q{\bf e}_1)(Q{\bf e}_1)^{\ast}\, :\, 0\neq \lambda\in\FF\}$ is a leaf of $\Psi(M_1)$ and hence a $\Psi$-image of some leaf $\{M_1+\lambda {\bf w}_1{\bf w}_1^{\ast}\, :\, 0\neq \lambda\in\FF\}$ of~$M_1$. Consequently, (\ref{ii16}) implies that $I_2=a {\bf e}_1 h_2({\bf w}_1)^{\ast}+\inv{a}  h_2({\bf w}_1){\bf e}_1^{\ast}$ holds for some $a\neq 0$, which is a contradiction, since the matrix on the right side of the equality has zero second diagonal entry.

Hence, $\rank\big(\Psi(M_1)-\Psi(M_2)\big)=1$, which concludes the proof of Step~3.
\end{myenumerate}

We are now able to end the proof. First, we apply Step~1 and Step~3 with $r=n$. We continue with a series of consecutively applications of Step~2 and Step~3. We start with $r=n-1$, proceed with $r=n-2$, etc. We end this procedure at $r=2$.  Finally we use Step~2 for $r=1$. In this way we obtain a (bijective)
map $\Psi: {\cal H}_n(\mathbb{F}_{q^2})\to {\cal H}_n(\mathbb{F}_{q^2})$ that preserves adjacency. By Lemma~\ref{FFA} (or by the fundamental theorem of geometry of hermitian matrices~\cite[Theorem~6.4]{wan}), $\Psi$ is of the form $\Psi(A)=PA^{\sigma}P^{\ast}+B$, where $B$ is some fixed hermitian matrix. Since $\Psi(0)=0$ by the construction, it follows that $B=0$. Consequently,
either $\Phi(A)=PA^{\sigma}P^{\ast}$ or $\Phi(A)=(P^{\ast})^{-1}(A^{\sigma})^{-1}P^{-1}$.
\end{proof}

\section{Minkowski space-time}

In this section we apply Main Theorem and authors previous result~\cite{FFA} for $2\times 2$ matrices to obtain  characterizations of maps that preserve the `speed of light' on (a) finite Minkowski space,  (b) the complement of the light cone in it. To better understand these results we firstly survey few related theorems and describe the connection with special theory of relativity.

A 4-dimensional Minkowski space-time $M_4$ is a vector space~$\mathbb{R}^4$ equipped with an indefinite inner product $$({\bf r}_1,{\bf r}_2):=-x_1x_2-y_1y_2-z_1z_2+c^2t_1t_2$$
between \emph{events} ${\bf r}_1:=(x_1,y_1,z_1,ct_1)^{\tr}$ and  ${\bf r}_2:=(x_2,y_2,z_2,ct_2)^{\tr}$. Here $c$ denotes the speed of light, and we choose units of measurement such that $c=1$. A $4\times 4$ matrix  $L$ is a \emph{Lorentz matrix} if $(L{\bf r}_1,L{\bf r}_2)=({\bf r}_1,{\bf r}_2)$ for all events ${\bf r}_1$ and ${\bf r}_2$. If
$$M=\begin{bmatrix}
-1& 0& 0& 0\\
0&-1 & 0& 0\\
0&0 &-1 &0 \\
0&0 &0 &1
\end{bmatrix},$$
then $({\bf r}_1,{\bf r}_2)={\bf r}_2^{\tr}M{\bf r}_1$ and we deduce that $L$ is a Lorentz matrix if and only if
\begin{equation}\label{eqq3}
L^{\tr} M L=M.
\end{equation}

Assume that to each point $(x,y,z)^{\tr}$ in $\RR^3$ a clock is assigned (to obtain $\RR^4$), which is synchronized with the clock at the origin $(0,0,0)^{\tr}$, that is, the event $(0,0,0,t)^{\tr}$ is `observed' from the point $(x,y,z)^{\tr}$ at time $t+\sqrt{x^2+y^2+z^2}/c$. In 1905 Einstein introduced special relativity~\cite{Einstein}, where he derived a Lorentz transformation, which transforms the coordinates of an event ${\bf r}=(x,y,z,t)^{\tr}$ from one synchronized system to coordinates $\psi({\bf r})=(\hat{x},\hat{y},\hat{z},\hat{t})^{\tr}$ of another synchronized system. In the derivation he assumed in particular that the motion between the two systems is uniform and linear, map $\psi$ is affine, and the speed of light is constant and equal in both systems. In 1950 Aleksandrov showed that the last assumption is sufficient. He proved the following theorem.
\begin{thm}\label{aleks1}
A bijective map $\psi: \RR^4\to \RR^4$ satisfies the rule
\begin{equation}\label{ee1}
({\bf r}_1-{\bf r}_2,{\bf r}_1-{\bf r}_2)=0 \Longleftrightarrow  \big(\psi({\bf r}_1)-\psi({\bf r}_2),\psi({\bf r}_1)-\psi({\bf r}_2)\big)=0
\end{equation}
if and only if it is of the form
$\psi({\bf r})=\alpha L {\bf r}+{\bf r}_0$
where $0\neq \alpha\in\RR$, ${\bf r}_0\in\RR^4$, and $L$ is a Lorentz matrix.
\end{thm}
Observe that $({\bf r}_1-{\bf r}_2,{\bf r}_1-{\bf r}_2)=0$ if and only if a light signal can pass between events ${\bf r}_1$ and ${\bf r}_2$, so maps that satisfy~(\ref{ee1}) are sometimes called maps that \emph{preserve the speed of light} (in both directions)~\cite{lester}.

In~\cite{zeeman}, Zeeman obtained a result similar to Theorem~\ref{aleks1}. He characterized bijective maps, which satisfy $({\bf r}_1-{\bf r}_2,{\bf r}_1-{\bf r}_2)=0$, $t_1< t_2$ if and only if $\big(\psi({\bf r}_1)-\psi({\bf r}_2),\psi({\bf r}_1)-\psi({\bf r}_2)\big)=0$, $\hat{t}_1< \hat{t}_2$. Moreover, in Lemma~1 he proved that these are the same bijective maps that satisfy $({\bf r}_1-{\bf r}_2,{\bf r}_1-{\bf r}_2)>0$, $t_1< t_2$ if and only if $\big(\psi({\bf r}_1)-\psi({\bf r}_2),\psi({\bf r}_1)-\psi({\bf r}_2)\big)>0$, $\hat{t}_1< \hat{t}_2$. Two events that satisfy $({\bf r}_1-{\bf r}_2,{\bf r}_1-{\bf r}_2)>0$ are \emph{time-like}, and for such events it is known that, in appropriate system, they appear at the same place at different time (a cause and its effect for example). So maps characterized by Zeeman are those maps from~(\ref{ee1}) that obey `causality', an assumption in special relativity, which says that an effect cannot occur before its cause in a different system.
It turns out that such maps are precisely those maps from Theorem~\ref{aleks1} for which $\alpha>0$ and the $(4,4)$-th entry of $L$ is $\geq 1$, or equivalently, $\alpha<0$ and the $(4,4)$-th entry of $L$ is $\leq - 1$~(cf.~\cite[p.~5]{naber}).

Hua reproved Theorem~\ref{aleks1} as an application of the fundamental theorem of geometry of $2\times 2$ complex hermitian matrices~\cite{hua_knjiga}. In fact, he considered the map
$\Omega: \RR^4\to\mathcal{H}_2(\CC)$, defined by
\begin{equation}\label{ee2}
\Omega((x,y,z,t)^{\tr})=\begin{bmatrix}
t+x& y+\imath z\\
y-\imath z& t-x
\end{bmatrix},
\end{equation}
and observed that
it is bijective with
\begin{equation}\label{ee5}
({\bf r}_1-{\bf r}_2,{\bf r}_1-{\bf r}_2)=0,\ {\bf r}_1\neq {\bf r}_2 \Longleftrightarrow \rank\big(\Omega({\bf r}_1)-\Omega({\bf r}_2)\big)=1.
\end{equation}
\v{S}emrl and Huang recently generalized the fundamental theorem for complex hermitian matrices~\cite{huang_semrl}. If the correspondence~(\ref{ee2}) and the techniques from~\cite{hua_knjiga} are applied to their result, the following is deduced.
\begin{thm}\label{semrl}
Let $\psi: \RR^4\to \RR^4$ be a map such that $\big(\psi({\bf r})-\psi({\bf r}^{\prime}),\psi({\bf r})-\psi({\bf r}^{\prime})\big)\neq 0$ for some ${\bf r}, {\bf r}^{\prime}\in\RR^4$. Then $\psi$
satisfies the rule
$$({\bf r}_1-{\bf r}_2,{\bf r}_1-{\bf r}_2)=0, {\bf r}_1\neq {\bf r}_2 \Longrightarrow  \big(\psi({\bf r}_1)-\psi({\bf r}_2),\psi({\bf r}_1)-\psi({\bf r}_2)\big)=0, \psi({\bf r}_1)\neq \psi({\bf r}_2)$$
if and only if it is of the form
$\psi({\bf r})=\alpha L {\bf r}+{\bf r}_0$,
where $0\neq \alpha\in\RR$, ${\bf r}_0\in\RR^4$, and $L$ is a Lorentz matrix.
\end{thm}
A different kind of generalization of Theorem~\ref{aleks1} was obtained by Aleksandrov~\cite{aleksandrov} and Lester~\cite{lester} (see also~\cite{lester_hand} and~\cite{popovici}). They proved the next result.
\begin{thm}\label{aleks-lester}
Let $\mathcal{D}\subseteq \RR^4$ be an open connected subset. A map $\varphi: \mathcal{D}\to \RR^4$ satisfies the rule
\begin{equation}\label{ee4}
({\bf r}_1-{\bf r}_2,{\bf r}_1-{\bf r}_2)=0 \Longleftrightarrow  \big(\varphi({\bf r}_1)-\varphi({\bf r}_2),\varphi({\bf r}_1)-\varphi({\bf r}_2)\big)=0
\end{equation}
if and only if it is a (well defined) composition of maps of the forms
\begin{equation}\label{ee3}
\varphi_1({\bf r})={\bf r}+{\bf r}_0,\ \varphi_2({\bf r})=\lambda {\bf r},\ \varphi_3({\bf r})=L {\bf r},\ \varphi_4({\bf r})=\frac{{\bf r}}{({\bf r},{\bf r})},\ \varphi_5({\bf r})=\frac{{\bf r}+({\bf r},{\bf r}){\bf n}}{1+2({\bf r},{\bf n})}.
\end{equation}
Here, ${\bf r}_0\in\RR^4$, $0\neq \lambda \in\RR$, $L$ is a Lorentz matrix, and ${\bf n}\in\RR^4$ satisfies $({\bf n},{\bf n})=0$.
\end{thm}
\begin{remark}
Theorems~\ref{aleks1} and~\ref{aleks-lester} have been proved for more general Minkowski spaces, however our interest in this paper is restricted to 4 dimensions.
\end{remark}
\begin{remark}\label{remark}
Maps $\varphi_4$ and $\varphi_5$ in~(\ref{ee3}) are not defined on the sets $C_0:=\{{\bf r}\in\RR^4 : ({\bf r},{\bf r})=0\}$ and $\{{\bf r}\in\RR^4 : ({\bf r},{\bf n})=-\frac{1}{2}\}$ respectively. Moreover, if $C_{{\bf r}^{\prime}}:=\{{\bf r}\in\RR^4 : ({\bf r}-{\bf r}^{\prime},{\bf r}-{\bf r}^{\prime})=0\}$, then routine calculations show that: (a) $\varphi_i(\RR^4\backslash C_{{\bf r}^{\prime}})=\varphi_i(\RR^4)\backslash \varphi_i(C_{{\bf r}^{\prime}})=\RR^4\backslash \varphi_i(C_{{\bf r}^{\prime}})$ for $i=1,2,3$; (b) $\varphi_1(C_{{\bf r}^{\prime}})=C_{{\bf r}^{\prime}+{\bf r}_0}$, $\varphi_2(C_{{\bf r}^{\prime}})=C_{\lambda{\bf r}^{\prime}}$, $\varphi_3(C_{{\bf r}^{\prime}})=C_{L{\bf r}^{\prime}}$; (c) $\varphi_4$ is not defined on whole $\RR^4\backslash C_{{\bf r}^{\prime}}$ unless ${\bf r}^{\prime}=0$; (d) $\varphi_4(\RR^4\backslash C_{0})=\RR^4\backslash C_{0}$; (e) if ${\bf r}^{\prime}$ is arbitrary and ${\bf n}\neq 0$, then $\varphi_5$ is not defined on whole $\RR^4\backslash C_{{\bf r}^{\prime}}$. It follows from these observations and Theorem~\ref{aleks-lester}, applied at $\mathcal{D}:=\RR^4\backslash C_0$, that a map $\varphi: \RR^4\backslash C_0\to \RR^4\backslash C_0$  satisfies~(\ref{ee4}) if and only if it is a composition of maps of the forms $\varphi_2, \varphi_3, \varphi_4$. If the
correspondence~(\ref{ee2}) and techniques from~\cite{hua_knjiga} are used in reversed way as above, we deduce that a map $\Phi: \mathcal{HGL}_2(\CC)\to \mathcal{HGL}_2(\CC)$ preserves adjacency in both directions if and only if it is of the form
$\Phi(A)=\lambda PA^{\sigma}P^{\ast}$ or $\Phi(A)=\lambda P(A^{\sigma})^{-1}P^{\ast}$, where $0\neq \lambda\in\RR$, $\sigma$ is either the identity map or complex conjugation, and complex matrix $P$ is invertible.
\end{remark}

We now switch from real Minkowski space-time to its finite analog. From now on $\FF_q$ is a finite field with $q$ elements such that $-1$ is not a square in $\FF_q$, that is, $q\equiv 3\ (\textrm{mod}\ 4)$ (cf.~\cite[p.~135]{wan2}). In particular,  $q$ is odd. Let $c=1\in\FF_q$ and assume that the product $({\bf r}_1, {\bf r}_2)$ between
${\bf r}_1, {\bf r}_2\in \FF_q^4$, a Lorentz matrix $L$, and $M$ are defined analogously as in the real space-time.
We say that a $4\times 4$ matrix $K$ over~$\FF_q$ is an \emph{anti-Lorentz matrix} if $(K{\bf r}_1,K{\bf r}_2)=-({\bf r}_1,{\bf r}_2)$ for all ${\bf r}_1,{\bf r}_2\in\FF_q^4$.

\begin{remark}
Anti-Lorentz matrices do not exist in the case of real Minkowski space-time. In fact, any such $K$ would satisfy $K^{\tr} M K=-M$, so $M$ and $-M$ would be congruent. This contradicts the Sylvester's law of inertia, since $M$ has 1 positive and 3 negative eigenvalues, while the opposite holds for $-M$.
\end{remark}

In 1972 Blasi et al.~\cite{blasi} obtained an analog of Theorem~\ref{aleks1} for finite Minkowski space-time over a prime field.
\begin{thm}\label{blasi}
Let $p>3$ be a prime with $p\equiv 3\ (\textrm{mod}\ 4)$. Then a bijective map $\psi: \FF_{p}^4\to \FF_p^4$ satisfies the rule
$$({\bf r}_1-{\bf r}_2,{\bf r}_1-{\bf r}_2)=0 \Longrightarrow  \big(\psi({\bf r}_1)-\psi({\bf r}_2),\psi({\bf r}_1)-\psi({\bf r}_2)\big)=0$$
if and only if it is of the form
\begin{equation*}
\psi({\bf r})=\alpha L {\bf r}+{\bf r}_0\qquad \textrm{or}\qquad \psi({\bf r})=\alpha K {\bf r}+{\bf r}_0,
\end{equation*}
where $0\neq \alpha\in\FF_{q}$ is a square, ${\bf r}_0\in\FF_q^4$, while $L$ and $K$ are Lorentz and anti-Lorentz matrices respectively.
\end{thm}

There exists an extensive literature of theoretical alternatives, where particle physics is not based over complex numbers, but over some other field instead. Some of these alternatives use finite fields (cf.~\cite{lev,blasi,foldes} and references therein). If $-1$ is not a square in $\FF_q$, then the field is of the form $\{0\}\cup\{\lambda_1^2,\ldots,\lambda_{(q-1)/2}^2\}\cup \{-\lambda_1^2,\ldots,-\lambda_{(q-1)/2}^2\}$. This motivated Blasi et al.~\cite{blasi} to interpret nonzero squares as `positive numbers' and non-squares  as `negative numbers' (note that the sum of two squares is not necessarily a square). With this interpretation anti-Lorentz matrices interchange the inner with the outer part of light-cones, that is, time-like events related to subluminal velocity are transformed to space-like events related to superluminal velocity and vice versa. Because of this the authors of reference~\cite{blasi} felt that transformations described in Theorem~\ref{blasi}
set up a favourable framework for the introduction of \emph{tachyons}, i.e., hypothetical particles that move faster than light, which were never discovered. It is not the purpose of this paper to either approve or criticize this physical interpretation. Below we generalize Theorem~\ref{blasi}. As mentioned by Hua for Einstein's special relativity~\cite[p.~92]{hua_knjiga}, a reduction of axioms/assumptions may help to either verify or overthrow a theory.

We first need few more technicalities. The splitting field of the polynomial $p(x)=x^2+1$, $p\in\FF_{q}[x]$, has $q^2$ elements (cf.~\cite[Corollary~2.15]{finite-fields-LN}). We denote it by~$\FFq2$. Since $x^q=x$ for all $x\in\FF_q$, $\FF_q=\FF$ is the fixed field of the involution $\inv{x}=x^q$ on $\FFq2$. Let $\imath\in\FFq2$ be such that $\imath^2=-1$. Since $q=4k+3$ for some integer $k$, it follows that $\inv{\imath}=(\imath^2)^{2k+1}\imath=-\imath$ and $\imath\inv{\imath}=1$. In particular, $x=\frac{x+\inv{x}}{2}+\imath\frac{x-\inv{x}}{2\imath}\in \FF+\imath \FF$ for all $x\in\FFq2$ and $\{1,\imath\}$ is a  basis of the vector space $\FFq2$ over the field $\FF_q$.
Let the map $\Omega: \FF_q^4\to\hdva$ be defined as in~(\ref{ee2}). Hence it is bijective and satisfies~(\ref{ee5}).
If $C_0:=\{{\bf r}\in\FF_q^4\ :\ ({\bf r},{\bf r})=0\}$, then $\Omega$ maps the set $\FF_q^4\backslash C_0$ bijectively onto $\hgldva$. Let $\omega$ be the restriction $\Omega|_{\FF_q^4\backslash C_0}$.

\begin{lemma}\label{minkowski1}
Assume that  $q\equiv 3\ (\textrm{mod}\ 4)$. Let $\psi=\Omega^{-1}\circ\Psi\circ\Omega$ and $\varphi=\omega^{-1}\circ\Phi\circ\omega$, where $\Psi: \hdva\to\hdva$, $\Phi: \hgldva\to\hgldva$ are defined as follows:
\begin{enumerate}
\item If $\Psi(A)=A+B$ for some $B\in \hdva$, then $\psi({\bf r})={\bf r}+{\bf r}_0$ for some ${\bf r}_0\in\FF_q^4$.
\item If $\Phi(A)=A^{-1}$, then $\varphi({\bf r})=\frac{L{\bf r}}{({\bf r},{\bf r})}$ for some Lorentz matrix~$L$.
\item If $\Psi(A)=A^{\sigma}$ for some field automorphism $\sigma: \FFq2\to\FFq2$, then $\psi({\bf r})=L{\bf r}^{\tau}$ for some Lorentz matrix $L$ and automorphism $\tau: \FF_q\to\FF_q$.
\item If $\Psi(A)=PAP^{\ast}$ for some invertible $P$ such that $\det P\inv{\det P}$ is a square in $\FF_q$, then $\psi({\bf r})=\alpha L{\bf r}$ for some Lorentz matrix $L$ and nonzero $\alpha\in\FF_q$.
\item If $\Psi(A)=PAP^{\ast}$ for some invertible $P$ such that $\det P\inv{\det P}$ is not a square in $\FF_q$, then $\psi({\bf r})=\alpha K{\bf r}$ for some anti-Lorentz matrix $K$ and  nonzero $\alpha\in\FF_q$.
\end{enumerate}
\end{lemma}
\begin{proof}
(i) If $B=[\begin{smallmatrix}\beta_1& b\\
\inv{b}& \beta_2\end{smallmatrix}]$, then ${\bf r}_0=\left(\frac{\beta_1-\beta_2}{2}, \frac{b+\inv{b}}{2},  \frac{b-\inv{b}}{2\imath}, \frac{\beta_1+\beta_2}{2}\right)^{\tr}$ fits the claim.

(ii) We can take $L=M$.

(iii) Recall that the automorphisms of a finite field with $p^m$ elements, where $p$ is a prime, are exactly the maps $x\mapsto x^{p^j}$, $0\leq j\leq m-1$ (cf.~\cite[Theorem~2.21]{finite-fields-LN}). Hence, $\sigma(\inv{x})=\inv{\sigma(x)}$ for all $x\in \FFq2$. Let $\sigma(\imath)=\alpha+\imath \beta$ for some $\alpha,\beta\in\FF_q$. Then equations $\imath+\inv{\imath}=0$ and $\imath\inv{\imath}=1$ imply that $\alpha=0$ and $\beta=\pm 1$, so $\sigma(\imath)\in\{\imath,-\imath\}$. If $\sigma(\imath)=\imath$, we choose the identity matrix as $L$, else if $\sigma(\imath)=-\imath$, we choose the diagonal matrix $\diag(1,1,-1,1)$ as $L$. In both cases $(\Omega^{-1}\circ\Psi\circ\Omega)({\bf r})=L{\bf r}^{\sigma}$. If $q=p^m$, then $\sigma(x)=x^{p^j}$ for some $0\leq j\leq 2m-1$. If $j\leq m-1$, then $\sigma$ restricted to $\FF_q$ is an automorphism $\tau$ of $\FF_q$. If $j=m+k$ for some $0\leq k\leq m-1$, then, for any $\lambda\in\FF_q$, $\sigma(\lambda)=\lambda^{p^j}=(\lambda^{q})^{p^k}=\lambda^{p^k}=:\tau(\lambda)$. Hence, in both cases $(\Omega^{-1}\circ\Psi\circ\Omega)({\bf r})=L{\bf r}^{\tau}$ for all ${\bf r}\in\FF_q^4$, where $\tau$ is an automorphism of $\FF_q$.

(iv) Let $\det P\inv{\det P}=\alpha^2$ for some $\alpha\in\FF_q$. Since the map $\N: \FFq2\to\FF_q$, given by $\N(x)=x\inv{x}$, is surjective (see e.g.~\cite{bose} or~\cite{finite-fields-LN}), there is $a\in\FFq2$ such that $\N(a)=\alpha^{-1}$. Then $\Psi(A)=\alpha Q A Q^{\ast}$ for $Q:=a P$ and $\det Q \inv{\det Q}=1$. Let ${\bf e}_1:=(1,0)^{\tr}$ and
${\bf e}_2:=(0,1)^{\tr}$. For any $B\in \hgldva$ define a column vector in $\FF_q^4$ by
$${\bf r}_B:=\left(\frac{{\bf e}_1^{\ast}B{\bf e}_1-{\bf e}_2^{\ast}B{\bf e}_2}{2},\frac{{\bf e}_1^{\ast}B{\bf e}_2+{\bf e}_2^{\ast}B{\bf e}_1}{2},\frac{{\bf e}_1^{\ast}B{\bf e}_2-{\bf e}_2^{\ast}B{\bf e}_1}{2\imath},\frac{{\bf e}_1^{\ast}B{\bf e}_1+{\bf e}_2^{\ast}B{\bf e}_2}{2}\right)^{\tr}.$$
For invertible $B=[\begin{smallmatrix}\beta_1&b\\
\inv{b}&\beta_2\end{smallmatrix}]$ and $C=[\begin{smallmatrix}\gamma_1&c\\
\inv{c}&\gamma_2\end{smallmatrix}]$ let $[\begin{smallmatrix}x_{11}&x_{12}\\
x_{21}&x_{22}\end{smallmatrix}]:=B^{-1}C$, i.e.,
$B[\begin{smallmatrix}x_{11}\\
x_{21}\end{smallmatrix}]=[\begin{smallmatrix}\gamma_1\\
\inv{c}\end{smallmatrix}]$ and $B[\begin{smallmatrix}x_{12}\\
x_{22}\end{smallmatrix}]=[\begin{smallmatrix}c\\
\gamma_2\end{smallmatrix}]$.
A short calculation and the Cramer's rule imply that
\begin{align}
\nonumber {\bf r}_B^{\tr}M{\bf r}_C&=\frac{1}{2}\left({\bf e}_1^{\ast}B{\bf e}_1{\bf e}_2^{\ast}C{\bf e}_2-{\bf e}_2^{\ast}B{\bf e}_1{\bf e}_1^{\ast}C{\bf e}_2+{\bf e}_2^{\ast}B{\bf e}_2{\bf e}_1^{\ast}C{\bf e}_1-{\bf e}_1^{\ast}B{\bf e}_2{\bf e}_2^{\ast}C{\bf e}_1\right)\\
\label{eqq1}&=\frac{1}{2}\left(\det[\begin{smallmatrix}\beta_1& c\\
\inv{b}& \gamma_2\end{smallmatrix}]+\det[\begin{smallmatrix}\gamma_1& b\\
\inv{c}& \beta_2\end{smallmatrix}]\right)\\
\nonumber &=\frac{1}{2}(x_{22}\det B+x_{11}\det B)=\frac{1}{2}\textrm{Trace}(B^{-1}C)\det{B},
\end{align}
where $\textrm{Trace}(X)$ denotes the sum of all diagonal entries of $X$. It is well known that map $\textrm{Trace}$ is similarity-invariant. Consequently,
\begin{align}
\nonumber {\bf r}_{QBQ^{\ast}}^{\tr}M{\bf r}_{QCQ^{\ast}}&=\frac{1}{2}\textrm{Trace}\big((Q^{\ast})^{-1}B^{-1}CQ^{\ast}\big)\det(QBQ^{\ast})\\
\label{eqq2}&=\frac{1}{2}\textrm{Trace}(B^{-1}C)\det{B}\det Q \inv{\det Q}={\bf r}_B^{\tr}M{\bf r}_C.
\end{align}
Let $L$ be the $4\times 4$ matrix with ${\bf r}_{B_j}$ as the $j$-th column, where
$$B_1:=Q[\begin{smallmatrix}1& 0\\
0& -1\end{smallmatrix}]Q^{\ast},\ B_2:=Q[\begin{smallmatrix}0& 1\\
1& 0\end{smallmatrix}]Q^{\ast},\ B_3:=Q[\begin{smallmatrix}0& \imath\\
\inv{\imath}& 0\end{smallmatrix}]Q^{\ast},\ B_4:=Q[\begin{smallmatrix}1& 0\\
0& 1\end{smallmatrix}]Q^{\ast}.$$
From (\ref{eqq2}) and (\ref{eqq1}) we deduce that
\begin{align*}
{\bf r}_{B_j}^{\tr}M{\bf r}_{B_k}&=0\quad (j\neq k),\\
{\bf r}_{B_j}^{\tr}M{\bf r}_{B_j}&=-1\quad (j=1,2,3),\\
{\bf r}_{B_4}^{\tr}M{\bf r}_{B_4}&=1,
\end{align*}
which is equivalent to (\ref{eqq3}), so $L$ is a Lorentz matrix. Let $\{{\bf f}_1,{\bf f}_2,{\bf f}_3,{\bf f}_4\}$ be the standard basis in $\FF_q^4$. Since $\Omega({\bf r}_B)=B$ for all $B$, we deduce that
\begin{equation}\label{eqq4}
(\Psi\circ\Omega)({\bf f}_j)=\Psi\big(Q^{-1}B_j(Q^{\ast})^{-1}\big)=\alpha B_j=\alpha \Omega({\bf r}_{B_j})=\alpha \Omega(L{\bf f}_j)=\Omega(\alpha L{\bf f}_j).
\end{equation}
The proof of (iv) ends by linearity and by composing equation (\ref{eqq4}) with $\Omega^{-1}$.

(v) Let $a\in\FFq2$ be such that $N(a)=-1$, $P_1:=[\begin{smallmatrix}1&0\\
0&a\end{smallmatrix}]$, and
$P_2:=PP_1^{-1}$. Then $\Psi=\Psi_2\circ\Psi_1$, where $\Psi_j(A)=P_j A P_j^{\ast}$. Since $-1$ is not a square in $\FF_q$, any non-square is a product of $-1$ and a square (cf.~\cite[Theorem~6.18]{wan2}). Hence, $\det P_2\inv{\det P_2}=-\det P\inv{\det P}$ is a square, and by (iv),
\begin{equation}\label{eqq5}
(\Omega^{-1}\circ\Psi_2\circ\Omega)({\bf r})=\alpha L{\bf r}
\end{equation}
for some Lorentz matrix $L$ and $\alpha\in\FF_q$. Write $a=\beta + \imath\gamma$ for some $\beta,\gamma\in\FF_q$. Then $N(a)=-1$ implies that $\beta^2+\gamma^2=-1$. Consequently, the matrix
$$K_2:=\begin{bmatrix}
0& 0& 0& 1\\
0&\beta & \gamma& 0\\
0&-\gamma &\beta &0 \\
1&0 &0 &0
\end{bmatrix}$$
is anti-Lorentz. Moreover, $(\Omega^{-1}\circ\Psi_1\circ\Omega)({\bf r})=K_2{\bf r}$ for all ${\bf r}$, so (\ref{eqq5}) shows that $(\Omega^{-1}\circ\Psi\circ\Omega)({\bf r})=\alpha LK_2{\bf r}$. Since $LK_2$ is an anti-Lorentz matrix, the proof ends.
\end{proof}

We are now able to apply author's previous result~\cite{FFA} to generalize Theorem~\ref{blasi} in two ways to obtain an analog of Theorem~\ref{semrl} for finite fields. That is, bijectivity of the map $\psi$ in Theorem~\ref{blasi} is reduced to a much weaker assumption, and the prime $p$ is replaced by a power of a prime.

\begin{thm}\label{minkowski2}
Let $q\equiv 3\ (\textrm{mod}\ 4)$. Then a map $\psi: \FF_{q}^4\to \FF_q^4$ satisfies the rule
$$({\bf r}_1-{\bf r}_2,{\bf r}_1-{\bf r}_2)=0, {\bf r}_1\neq {\bf r}_2 \Longrightarrow  \big(\psi({\bf r}_1)-\psi({\bf r}_2),\psi({\bf r}_1)-\psi({\bf r}_2)\big)=0, \psi({\bf r}_1)\neq \psi({\bf r}_2)$$
if and only if it is of the form
\begin{equation}\label{eqq6}
\psi({\bf r})=\alpha L {\bf r}^{\tau}+{\bf r}_0\qquad \textrm{or}\qquad \psi({\bf r})=\alpha K {\bf r}^{\tau}+{\bf r}_0,
\end{equation}
where $0\neq \alpha\in\FF_{q}$ and ${\bf r}_0\in\FF_q^4$ are fixed, $\tau$ is an automorphism of $\FF_q$, while $L$ and~$K$ are Lorentz and anti-Lorentz matrices respectively.
\end{thm}
\begin{proof}
It is straightforward to check that maps~(\ref{eqq6}) obey the rule in the Theorem~\ref{minkowski2}. From Lemmas~\ref{FFA} and~\ref{minkowski1} we deduce that these are the only such maps.
\end{proof}

If we apply Main Theorem, we deduce a result that is similar to the one in Remark~\ref{remark} for real Minkowski space-time, but it is in a sense much stronger, since equivalence (\ref{ee4}) is essentially replaced by an implication.

\begin{thm}\label{minkowski3}
Let $q\equiv 3\ (\textrm{mod}\ 4)$ and $q\neq 3$.
Then a map $\varphi: \FF_{q}^4\backslash C_0\to \FF_q^4\backslash C_0$ satisfies the rule
$$({\bf r}_1-{\bf r}_2,{\bf r}_1-{\bf r}_2)=0, {\bf r}_1\neq {\bf r}_2 \Longrightarrow  \big(\varphi({\bf r}_1)-\varphi({\bf r}_2),\varphi({\bf r}_1)-\varphi({\bf r}_2)\big)=0, \varphi({\bf r}_1)\neq \varphi({\bf r}_2)$$
if and only if it fits one of the following four forms:
\begin{equation}\label{eqq7}
\varphi({\bf r})=\alpha L {\bf r}^{\tau},\quad
\varphi({\bf r})=\alpha K {\bf r}^{\tau},\quad
\varphi({\bf r})= \frac{\alpha L {\bf r}^{\tau}}{({\bf r}^{\tau},{\bf r}^{\tau})},\quad
\varphi({\bf r})=\frac{\alpha K {\bf r}^{\tau}}{({\bf r}^{\tau},{\bf r}^{\tau})}.
\end{equation}
where $0\neq \alpha\in\FF_{q}$ is fixed, $\tau$ is an automorphism of $\FF_q$, while $L$ and $K$ are Lorentz and anti-Lorentz matrices respectively.
\end{thm}
\begin{proof}
It is straightforward to check that maps~(\ref{eqq7}) obey the rule in Theorem~\ref{minkowski3}. Main Theorem and Lemma~\ref{minkowski1} show that these are the only such maps.
\end{proof}

\begin{remark}
For prime fields, the identity map is the only automorphism $\tau$ of $\FF_p$.
\end{remark}

\begin{remark}
Formulations of Theorem~\ref{minkowski2}, Theorem~\ref{minkowski3}, and Lemma~\ref{minkowski1} (iv), (v) are still valid if the expression `$0\neq\alpha\in\FF_q$' is replaced by `$\alpha\in\FF_q$ is a nonzero square' or by `$\alpha\in\FF_q$ is non-square'. This is true, since $-1$ is non-square and so precisely one of the scalars $\alpha,-\alpha$ is a square. Since $-I$ is a Lorentz matrix, we can replace the pair $(\alpha, L)$ by $(-\alpha, -L)$ and the pair $(\alpha, K)$ by $(-\alpha, -K)$.
\end{remark}

\begin{remark}
In contrast to Theorem~\ref{minkowski2}, there is an assumption in Theorem~\ref{semrl} on existence of ${\bf r}, {\bf r}^{\prime}$ such that $\psi({\bf r})-\psi({\bf r}^{\prime})\notin C_0$. In fact, for arbitrary infinite field~$\KK$ there always exist degenerate maps $\psi: \KK^4\to\KK^4$ and  $\psi: \KK^4\backslash C_0\to\KK^4\backslash C_0$ that obey the rule from Theorem~\ref{minkowski2}/Theorem~\ref{minkowski3} and which satisfy
$\psi({\bf r})-\psi({\bf r}^{\prime})\in C_0$
for all ${\bf r}, {\bf r}^{\prime}$. An example of a such map is $\psi({\bf r}):=\big(-g({\bf r}),1,0,g({\bf r})\big)^{\tr}$, where $g:\KK^4\to\KK$ is an injection. Recall that for Theorem~\ref{minkowski3},
the non-existence of such maps in finite field case was essentially proven in Lemma~3.8 of the first paper~\cite{prvi_del}, where spectral graph theory and chromatic number of a graph were used.
\end{remark}

\end{document}